\documentclass[bj]{imsart}

\RequirePackage{amsthm,amsmath,amssymb}
\RequirePackage[numbers]{natbib}

\startlocaldefs
\numberwithin{equation}{section}
\theoremstyle{plain}
\newtheorem{thm}{Theorem}[section]

\newtheorem{lemma}{Lemma}[section]
\newtheorem{remark}{Remark}[section]

\endlocaldefs

\begin{document}
\global\long\def\xib{\boldsymbol{\xi}}
\global\long\def\etab{\boldsymbol{\eta}}
\global\long\def\thetab{\boldsymbol{\theta}}
\global\long\def\yb{\boldsymbol{y}}
\global\long\def\eb{\boldsymbol{e}}
\begin{frontmatter}
\title{On consistency of nonparametric tests}
\runtitle{ On consistency of nonparametric tests}

\begin{aug}
\author{\fnms{Mikhail} \snm{Ermakov}\ead[label=e1]{erm2512@gmail.com}}

\runauthor{M. Ermakov}

\address{Institute of Problems of Mechanical Engineering RAS, Bolshoy pr., 61, VO, 1991178 RUSSIA and
St. Petersburg State University, Universitetsky pr., 28, Petrodvoretz, 198504 St. Petersburg, RUSSIA\\
\printead{e1}}

\end{aug}

\begin{abstract} For   $\chi^2-$tests with increasing number of cells, Cramer-von Mises tests, tests generated  $\mathbb{L}_2$- norms of kernel estimators and  tests generated quadratic forms of estimators of Fourier coefficients, we find  necessary and sufficient conditions of consistency  and inconsistency for sequences  of alternatives having a given rate of convergence to hypothesis in $\mathbb{L}_2$-norm. We provide transparent interpretations of these conditions allowing to understand the structure of such consistent sequences. For problem of signal detection in Gaussian white noise we show that, if set of alternatives is bounded closed center-symmetric convex set $U$ with deleted "small" $\mathbb{L}_2$ -- ball, then compactness of set $U$ is necessary condition for existence of consistent tests.
  \end{abstract}



\end{frontmatter}

\section{Introduction}

 For exploration of problem of nonparametric hypothesis testing on a density a priori information is introduced usually that density belongs to some set $U$ of smooth or convex, or ... functions (see \cite{dal}, \cite{er90}, \cite{ing02},  \cite{ing12},
  \cite{la},    \cite{lep}   and references therein). The same situation takes place for problem of signal detection in Gaussian noise. In paper we explore the problem of choice of largest sets $U$ for these setups  and propose new setup allowing  to explore the problem without introduction of such a priori information.

We answer on the following questions.

 {\it For which largest sets   $U$ are there   uniformly consistent tests?}

 This problem is explored for signal detection in Gaussian white noise. We show that, if set of alternatives is closed bounded center-symmetric convex set $U$ with deleted "small" $\mathbb{L}_2$-ball, then  uniformly consistent tests exist, iff, the set $U$ is compact. This statement shows that such a setup requires significant a priori information on sets of alternatives. Note that, for existence of uniformly consistent nonparametric estimators, the compactness is also necessary and sufficient condition (see  \cite{ih}  and  \cite{jo}). Similar statement holds also in theory of ill-posed inverse problems with deterministic noise \cite{en}. Problem of consistency of tests has been explored in many papers and for different setups. Rather complete bibliography one can find in \cite{er15}.

 {\it
Let test statistics be given and let rate of convergence to zero for radius of deleted "small" balls be known. What is  largest sets $U$ in this setup?}

 Such sets we call  maxisets. For   $\chi^2$-tests with increasing number of cells, Cramer -von Mises tests, tests generated  $\mathbb{L}_2$- norms of kernel estimators and  tests generated quadratic forms of estimators of Fourier coefficients (Theorem \ref{tq1}) we show  that maxisets are Besov bodies $B^s_{2\infty}(P_0)$, $P_0  > 0$.

  All  above mentioned test statistics are quadratic functionals. This allows to develop unified approach to exploration and to prove similar results for all these setups.

 For  nonparametric estimation the notion of maxisets has been introduced Kerkyacharian and Picard \cite{ker93}. Maxisets of  nonparametric  estimators have been comprehensively explored in \cite{co},   \cite{ker02},  \cite{rio} (see also references therein). For nonparametric hypothesis testing completely different definition of maxisets has been introduced  Autin, Clausel,  Freyermuth and  Marteau \cite{au}.

  Chi-squared tests and Cramer-von Mises tests are  explored for the problem of hypothesis testing on a density of distribution.

Let $X_1,\ldots,X_n$ be i.i.d.r.v.'s with c.d.f. $F(x)$, $x \in (0,1)$. Let c.d.f. $F(x)$ have a density $p(x) = 1 + f(x) = dF(x)/dx$. Suppose  $f \in \mathbb{L}_2(0,1)$ with the norm
$
\|f\| = \left(\int_0^1 f^2(x) dx \right)^{1/2} < \infty.
$

One needs to verify hypothesis
\begin{equation}\label{i1}
\mathbb{H}_0\,:\, f(x) = 0, \quad x \in (0,1),
\end{equation}
versus $f$ belongs to some nonparametric set of alternatives.

Tests generated  $\mathbb{L}_2$- norms of kernel estimators and
 tests generated quadratic forms of estimators of Fourier coefficients are explored for problem of signal detection in Gaussian white noise.
  We observe a realization of random process $Y_n(t)$ defined stochastic differential equation
\begin{equation}\label{q1}
dY_n(t) = f(t) dt + \frac{\sigma}{\sqrt{n}}\, dw(t), \quad t \in [0,1],\quad \sigma >0,
\end{equation}
where  $f \in \mathbb{L}_2(0,1)$  is  unknown signal and $dw(t)$ is Gaussian white noise.

 The answers on two previous questions are provided for the following setup. We have a priori information that function $f$ belongs to a ball $U$  in some functional space $\Im$. We wish to test  hypothesis (\ref{i1})  versus alternatives
\begin{equation}\label{i26}
\mathbb{H}_n\, : \, f \in V_n = \{ f \,: \, \|f\|^2 \ge \rho_n,\, f \in U\, \}
\end{equation}
with $\rho_n \to 0$ as $n \to \infty$.

We show that there is sequence $\rho_n \to 0$ as $n \to \infty$ such that consistent tests exist, iff, ball $U$ is compact in $\mathbb{L}_2$  (see Theorems \ref{tqq} and \ref{tqq1}).

The answer on the second question is provided for $\rho_n \asymp n^{-2r}$,  $0< r< 1/2$. For such a choice $r$ we have $ r = \frac{2s}{1 + 4s}$ for
 $\chi^2-$tests with increasing number of cells, tests generated  $\mathbb{L}_2$- norms of kernel estimators, tests generated quadratic forms of estimators of Fourier coefficients and  $r = \frac{s}{2+2s}$ for Cramer-von Mises tests.
   Uniform consistency of chi-squared tests  and Cramer - von Mises tests on sets  $V_n$ if set $U$ is above mentioned Besov balls has been established Ingster \cite{ing87}.

   Sequence of alternatives may be consistent, has given rate of convergence to hypothesis and does not belong to maxisets. Thus   sets of alternatives $V_n$  cover only a part of all consistent alternatives.

   {\it How to describe, for given test statistics, all consistent and inconsistent sequences of alternatives with fixed rates of convergence to hypothesis in $\mathbb{L}_2$-norm?}

       We explore the problem of hypothesis testing (\ref{i1}) versus  alternatives
     \begin{equation}\label{i30}
\mathbb{H}_n\,:\, f = f_n, \qquad cn^{-r} \le \| f_n\| \le C n^{-r}, \qquad  0 < r < 1/2.
\end{equation}
          For above mentioned test statistics  answer on this question is provided in terms of concentration of Fourier coefficients (Theorems \ref{tq3} and \ref{tq4}). We propose the following interpretation (Theorem \ref{tq7}) of  these results:

  {\it functions $f_n$ of consistent sequence of alternatives having given rate of convergence to hypothesis admit representation as
functions $f_{1n}$ from maxiset with the same rate of convergence to hypothesis plus orthogonal functions $f_{n} - f_{1n}$.}

If we suppose that the smoothest part of the alternatives belongs to maxiset then we can provide the following interpretation of this statement.

\
  {\it Any consistent sequence of alternatives having given rate of convergence to hypothesis   admits representation as smooth
functions from maxiset with the same rate of convergence to hypothesis plus orthogonal more oscillating functions.}

We show   (Theorem   \ref{tq11}) that, for any $\varepsilon > 0$, there are maxiset and functions $f_{1n}$  from maxiset such that the differences of type II error probabilities for alternatives $f_{n}$ and $f_{1n}$ is smaller   $\varepsilon $.

    Thus, each  function of consistent sequence of alternatives  with fixed rate of convergence to hypothesis  contains sufficiently smooth function  as an additive component and this function carries almost all information on its type II error probability.

  {\it What can we say about properties of consistent and inconsistent sequences of alternatives having fixed rate of convergence to hypothesis in $\mathbb{L}_2$- norm?}

  We show (Theorem \ref{tq5}) that asymptotic of type II error probabilities of sums of alternatives from consistent and inconsistent sequences coincides with the asymptotic for consistent sequence.

We call sequence of alternatives $f_n$ purely consistent if we could not distinguish from this sequence inconsistent sequence of alternatives $f_{2n}$ having the same rates of convergence to hypothesis and such that $f_{2n}$ are orthogonal to $f_n  - f_{2n}$. In terms of concentration of Fourier coefficients we point out (see Theorem \ref{tq6}) analytic assignment of purely consistent sequences of alternatives.
    It is easy to  show that any sequence of alternatives from maxisets with fixed rates of convergence to hypothesis is purely consistent.

  We show (Theorem \ref{tq12}) that, for any $\varepsilon>0$, for any purely consistent sequence of alternatives $f_n $, $cn^{-r} \le \| f_n\| \le C n^{-r}$,  there are maxiset and some sequence  $f_{1n}$ from this maxiset, such that there holds $\|f_n - f_{1n}\| \le \varepsilon n^{-r}$.
\vskip 0.15cm
 Paper is organized as follows. In section \ref{sec2} main definitions are introduced. In section \ref{sec3},  the answer on the first question is provided. In sections  \ref{sec4}, \ref{sec5}, \ref{sec6} and \ref{sec7}  above mentioned  results are established respectively for  test statistics based on quadratic forms of estimators of Fourier coefficients,  $\mathbb{L}_2$ -- norms of kernel estimators, $\chi^2$--tests and Cramer-- von Mises tests. Proof of all Theorems is provided in Appendix.

  Exploration of consistency for test statistics based on quadratic forms of estimators of Fourier coefficients,  $\mathbb{L}_2$--norms of kernel estimators and  $\chi^2$--tests with increasing number of cells is based on Theorems (see Theorems  \ref{tq2}, \ref{tk2} and \ref{chi2}) on asymptotic minimaxity of these test statistics in semiparametric setup. In semiparametric setup (distance method) sets of alternatives are defined distance generating test statistics. Set of alternatives is set of all alternatives such that their distance from hypothesis is more then given constant. These results (see Theorems  \ref{tq2}, \ref{tk2} and \ref{chi2}) reduce the exploration of consistency of alternatives  to the exploration of rates of convergence of distances of alternatives from hypothesis.  For Cramer-von Mises test statistics  a similar statement  has not been established. Thus, in section 7, we prove uniform consistency of Cramer -von Mises test statistics on sets of alternatives such that normalized Cramer -von Mises distances of these alternatives from hypothesis are more than some positive constant (see Theorem \ref{tcm}).

We use letters $c$ and $C$ as a generic notation for positive constants. Denote ${\bf 1}_{\{A\}}$ the
indicator of an event $A$.  Denote $[a]$ whole part of real number $a$. For any two sequences of positive real numbers $a_n$ and $b_n$,  $a_n \asymp b_n$ implies $c < a_n/b_n < C$ for all $n$ and $a_n = o(b_n)$ implies $a_n/b_n \to 0$ as $n \to \infty$. For any complex number $z$ denote $\bar z$ complex conjugate number.

Denote
$$ \Phi(x) = \frac{1}{\sqrt{2\pi}}\,\int_{-\infty}^x\,\exp\{-t^2/2\}\, dt, \quad x \in \mathbb{R}^1,
$$
the standard normal distribution function.

Let $\phi_j$, $1 \le j < \infty$, be orthonormal system of functions onto $\mathbb{L}_2(0,1)$. Define the sets
\begin{equation}\label{vv}
\mathbb{\bar B}^s_{2\infty}(P_0) = \Bigl\{f : f = \sum_{j=1}^\infty\theta_j\phi_j,\,\,\,  \sup_{\lambda>0} \lambda^{2s} \sum_{j>\lambda} \theta_j^2 \le P_0,\,\, \theta_j \in \mathbb{R}^1 \Bigr\}.
\end{equation}
Under some conditions on the basis $\phi_j, 1 \le j < \infty,$  the space
$$
\bar{\mathbb{ B}}^s_{2\infty} = \Bigl\{ f : f = \sum_{j=1}^\infty\theta_j\phi_j,\,\,\,  \sup_{\lambda>0} \lambda^{2s} \sum_{j>\lambda}\, \theta_j^2 < \infty,\,\, \theta_j \in \mathbb{R}^1 \Bigr\}
$$
is Besov space $\mathbb{B}^s_{2\infty}$ (see   \cite{rio}).
In particular, $\mathbb{\bar B}^s_{2\infty}$ is Besov space  if $\phi_j$, $1 \le j < \infty$, is trigonometric basis.

If $\phi_j(t) = \exp\{2\pi i j x\}$, $x\in (0,1)$, $j = 0, \pm 1, \ldots$,  denote
$$
\mathbb{ B}^s_{2\infty}(P_0) = \Bigl\{f : f = \sum_{j=-\infty}^\infty \theta_j\phi_j,\,\,\,  \sup_{\lambda>0} \lambda^{2s} \sum_{|j| >\lambda} |\theta_j|^2 \le P_0 \Bigr\}.
$$
Here $\theta_j$ are complex numbers  and $\theta_j =  \bar\theta_{-j}$ for all $-\infty < j < \infty$.

For the same basis denote
$$
\mathbb{\tilde B}^s_{2\infty}(P_0) = \Bigl\{f : f = \sum_{j=-\infty}^\infty \theta_j\phi_j,\,\,f \in \mathbb{ B}^s_{2\infty}(P_0),\,  \theta_0 =0\, \Bigr\}.
$$
The balls in Nikols'ki classes
$$
\int\,(f^{(l)}(x+t) - f^{(l)}(x))^2\, dx \le L |t|^{2(s-l)}, \quad \|f\| < C
$$
with $l = [s]$ are the  balls in $\mathbb{B}^s_{2\infty}$.
\section{Main definitions \label{sec2}}
\subsection{ Consistency and $n^{-r}$-consistency \label{ss2.1}}
For any test $K_n$  denote $\alpha(K_n)$ its type I error probability, and $\beta(K_n,f)$ its type II error probability for  alternative $f \in \mathbb{L}_2(0,1)$.

Definition of consistency  will be different in each  section. In section \ref{sec3} we explore the problem of existence of consistent tests and consistency is considered among all tests.

In section \ref{sec4} consistency is considered for a given sequence of test statistics $T_n$. For kernel-based tests and chi-squared tests, consistency is explored for whole population of test statistics depending on kernel width and number of cells respectively. In section \ref{sec7} we have only one test statistic.

Below we provide the definition of consistency for setup of \ref{sec4}.  Thus
all definition of further subsections can be considered only for this setup. However these definition are valid for setups of sections \ref{sec5} - \ref{sec7}.

We say that sequence of alternatives $f_n$ is {\sl consistent}  if for any $\alpha$, $0 < \alpha < 1$, for sequence of tests $K_n$, $\alpha(K_n) = \alpha\,(1 + o(1))$, generated test statistics $T_n$, there holds
\begin{equation}\label{vas1}
\limsup_{n\to\infty}  \beta(K_n, f_n) < 1 - \alpha.
\end{equation}
If $cn^{-r} < \|f_n\| < Cn^{-r}$ additionally, we say that sequence of alternatives $f_n$ is $n^{-r}$- {\sl consistent} (see  \cite{ts}).

We say that sequence of alternatives $f_n$ is {\sl inconsistent}  if, for each sequence of tests $K_n$ generated test statistics $T_n$, there holds
\begin{equation}\label{vas25}
\liminf_{n\to\infty} (\alpha(K_n) + \beta(K_n, f_n)) \ge 1.
\end{equation}

Suppose we consider problem of testing hypothesis  $\mathbb{H}_0 : f =0$ versus alternative
\begin{equation}\label{uc1}
\mathbb{H}_n : f \in \Psi_n,
\end{equation}
where $\Psi_n$ are bounded subsets of $\mathbb{L}_2(0,1)$.

For tests $K_n$, $\alpha(K_n) = \alpha + o(1)$, $0 < \alpha <1$, generated test statistics $T_n$ denote
$\beta(K_n,\Psi_n) = \sup_{f \in \Psi_n} \beta(K_n,f)$.

We say that sets $\Psi_n$ of alternatives are uniformly consistent if
\begin{equation}\label{uc1}
\limsup_{n \to \infty} \beta(K_n,\Psi_n) < 1 - \alpha.
\end{equation}
 Set $\Psi_n$ is bounded subset of $\mathbb{L}_2(0,1)$. Therefore $\Psi_n$ is compact in weak topology in  $\mathbb{L}_2(0,1)$. Hence it is easy to show that sequence of sets $\Psi_n$ is uniformly consistent, if and only if, sets $\Psi_n$ do not contain inconsistent sequence of alternatives $f_n \in \Psi_n$.  In other words, sequence of sets $\Psi_n$ is uniformly consistent, if and only if, all sequences of alternatives $f_n \in \Theta_n$ are consistent. Thus the problem on consistency on sets of alternatives is reduced to the problem of consistency on sequences of alternatives.
\subsection{Purely consistent sequences}
We say that $n^{-r}$- consistent sequence of alternatives  $f_n$ is {\sl purely $n^{-r}$-consistent} if there does not exist subsequence $f_{n_i}$ such that $f_{n_i} = f_{1n_i} + f_{2n_i}$ where  $f_{2n_i}$ is orthogonal to  $f_{1n_i}$ and sequence $f_{2n_i}$, $\|f_{2n_i}\| > c_1n^{-r}$, is inconsistent.
\subsection{Maxisets \label{ss2.3}}
Let $\phi_j$, $1 \le j < \infty$, be orthonormal basis in $\mathbb{L}_2(0,1)$. We say that a set $U$, $U \subset \mathbb{L}_2(0,1)$, is ortho-symmetric with respect to this basis if $f = \sum_{j=1}^\infty \theta_j \phi_j \in U$ implies $\tilde f = \sum_{j=1}^\infty \tilde\theta_j \phi_j \in U$ for any $\tilde\theta_j = \theta_j$ or $\tilde\theta_j = -\theta_j$, $j=1,2,\ldots$.

For closed convex bounded set $U \subset \mathbb{L}_2(0,1)$ denote $\Im_U$ functional space with unite ball $U$.

For the problem of signal detection we call bounded closed  set $\gamma U \subset \mathbb{L}_2(0,1)$,  {\sl maxiset} if
\vskip 0.15cm
{\sl i.} set $U$ is  convex,
\vskip 0.15cm
{\sl ii.}    the set $U$ is ortho-symmetric with respect to  orthonormal basis $\phi_j$, $1 \le j < \infty$,
\vskip 0.15cm
{\sl iii.} any subsequence of alternatives $f_{n_i} \in \gamma\,U$, $cn_i^{-r} < \|f_{n_i}\| < Cn_i^{-r}$, $n_i \to \infty$ as $i \to \infty$, is consistent,
\vskip 0.25cm
{\sl iv.} if $f \notin \Im_U$, then, in any convex, ortho-symmetric set $V$ that contains  $f$, there is inconsistent subsequence  of alternatives $f_{n_i} \in V$, $cn_i^{-r} < \|f_{n_i}\| < Cn_i^{-r}$, where $n_i \to \infty$ as $i \to \infty$.
\vskip 0.15cm
{\sl iv.} implies that $U$ is the largest set satisfying {\sl i.- iii.}

For problem of hypothesis testing on a density, in definition of maxiset we make additional assumption:
\vskip 0.15cm
{\sl iv.} is considered only for   functions $f = 1 + \sum_{i=1}^\infty \theta_i \phi_i$ (or $f = 1 + \sum_{|i| \ge 1}^\infty \theta_i \phi_i$) satisfying the following condition.
 \vskip 0.15cm
 {\bf D.} There is $l_0 = l_0(f)$ such that, for all $l > l_0$,  functions $1 + \sum_{|i| >l}^\infty \theta_i \phi_j$   are nonnegative (are densities).

 D allows to analyze tails $f_{n_j} = \sum_{|i| \ge j} \theta_i \phi_i$ to establish {\sl iv.}

 If $U$ is maxiset, then $\gamma U$, $0 < \gamma < \infty$, is maxiset as well.

 Simultaneous assumptions of convexity and ortho-symmetry of set $V$ is rather strong. If $f \in V$, $f = \sum_{i=1}^\infty \theta_i \phi_i$, then any $f_\eta \in V$ with  $f_\eta = \sum_{i=1}^\infty \eta_i \phi_i$,  $|\eta_i| < |\theta_i|$, $1 \le i < \infty$.

Test statistics of tests generated  $\mathbb{L}_2$- norms of kernel estimators and   Cramer-von Mises tests  admit representation as a linear combination of squares of estimators of Fourier  coefficients. Therefore, for  these test statistics, consistency of sequence $f_n $ implies consistency of any sequence of ortho-symmetric functions $\tilde f_n$ generated $f_n$. Moreover, type II error probabilities of sequences $f_n$ and $\tilde f_{n}$ have the same asymptotic. Thus  the requirement {\sl ii.} seems natural for test statistics admitting representation as a liner combination of squares of estimators of Fourier  coefficients.
For chi-squared tests, by Theorem \ref{tchi3} given in what follows, the same statement holds.
\subsection{Another approach to definition of maxisets \label{ss2.5}}
Requirement of ortho-symmetry of set $U$ does not allow to call maxiset any convex set $W$ generated equivalent norm in $\Im_U$. In definition given below we do not make such an assumption.

In this definition of maxiset we do not suppose ortho-symmetry of set $U$.

Let $\Im \subset \mathbb{L}_2(0,1)$ be Banach space with a norm $\|\cdot\|_\Im$. Denote $\gamma U=\{f:\, \|f\|_\Im \le \gamma,\, f \in \Im\}$, $\gamma > 0,$  a ball in $\Im$.

Define subspaces $\Pi_k$, $1 \le k < \infty$, by induction.

Denote $d_1= \max\{\|f\|,\, f \in U\}$ and denote $e_1$ function $e_1 \in U$ such that $\|e_1\|= d_1.$ Denote $\Pi_1$ linear space generated vector $e_1$.

For $i=2,3,\ldots$ denote
$d_i = \max\{\rho(f,\Pi_{i-1}), f \in U \}$ with $\rho(f,\Pi_{i-1})=\min\{\|f-g\|, g \in \Pi_{i-1} \}$. Define function $e_i$, $e_i \in U$, such that $\rho(e_i,\Pi_{i-1}) = d_i$.
Denote $\Pi_i$ linear space generated functions $e_1,\ldots,e_i$.

For any $f \in \mathbb{L}_2(0,1)$ denote
$f_{\Pi_i}$ the projection of  $f$ onto the subspace $\Pi_i$ and denote $\tilde f_i = f - f_{\Pi_i}$.

Thus we associate with each $f \in \mathbb{L}_2(0,1)$ sequence of functions $\tilde f_i, \tilde f_i \to 0$ as $i \to \infty$. This allows to cover by  our consideration  all space $\mathbb{L}_2(0,1)$. Suppose that the functions $e_1,e_2,\ldots$ are sufficiently smooth. Then, considering the functions $\tilde  f_i = f - f_{\Pi_i}$, we "in some sense delete the most smooth part $f_{\Pi_i}$ of function $f$ and explore the behaviour of remaining part."

For the problem of signal detection we say that  set $U$ is maxiset for test statistics $T_n$ and $\Im$ is maxispace if the following two statements take place.
\vskip 0.3cm
{\sl i.} any subsequence of alternatives $f_{n_j} \in U$, $cn_j^{-r} < \|f_{n_j}\| < Cn_j^{-r}$, $n_j \to \infty$ as $j \to \infty$, is consistent,.
\vskip 0.3cm
{\sl ii.}  for any  $f \in \mathbb{L}_2(0,1)$, $f \notin \Im$,  there are sequences  $i_n, j_{i_n}$ with $i_n \to \infty$ as $n \to \infty$   such that $c j_{i_n}^{-r}<\|\tilde f_{i_n}\| < C j_{i_n}^{-r}$ for some constants $c$ and  $C$ and subsequence $\tilde f_{i_n}$ is $j_{i_n}^{-r}$- inconsistent.
\vskip 0.3cm
For problem of hypothesis testing on a density  we make additional requirement in {\sl ii.} that $1+ \tilde f_{i_n}$ should be the densities.

  We provide proofs of Theorems for definition of maxisets in terms of subsection \ref{ss2.3}. However it is easy to see that slight modification of this reasoning  provide proofs for definition of maxisets of subsection \ref{ss2.5} as well.

\section{Necessary and sufficient conditions of consistency \label{sec3}}
   We  consider  problem of signal detection in Gaussian white noise discussed in Introduction. Problem is explored in terms of sequence model.

 The stochastic differential equation (\ref{q1}) can be rewritten  in terms of  a sequence model based on orthonormal system of functions $\phi_j$, $1 \le j < \infty$, in the following form
\begin{equation}\label{q2}
y_j = \theta_j + \frac{\sigma}{\sqrt{n}} \xi_j, \quad 1 \le j < \infty,
\end{equation}
where $$y_j = \int_0^1 \phi_j dY_n(t), \quad \xi_j = \int_0^1\,\phi_j\,dw(t) \quad \mbox{ and}  \quad \theta_j = \int_0^1 f\,\phi_j\,dt.$$ Denote $\yb =  \{y_j\}_{j=1}^\infty$ and $\thetab = \{\theta_j\}_{j=1}^\infty$.

 We can consider $\thetab$ as a vector in Hilbert space $\mathbb{H}$ with the norm $\|\thetab\| = \Bigl(\sum_{j=1}^\infty \theta_j^2\Bigr)^{1/2}$. We implement the same notation $\| \cdot \|$ in $\mathbb{L}_2$ and in $\mathbb{H}$. The sense of this notation will be always clear from context.

In this notation the problem of hypothesis testing can be rewritten in the following form.
One needs to test the hypothesis $\mathbb{H}_0 : \thetab = 0$ versus alternatives $\mathbb{H}_n : \thetab \in V_n =\{\, \theta : \|\thetab\| \ge \rho_n,\, \thetab \in U,\, U \subset \mathbb{H}\,\}$.

Denote
$$
\beta(K_n,V_n) = \sup\{ \beta(K_n,f), f \in V_n\}.
$$
We say that there is consistent sequence of tests for sets of alternatives (\ref{i26})
 if there is sequence of tests $K_n$  such that
\begin{equation}\label{uux}
\limsup_{n\to\infty} (\alpha(K_n) + \beta(K_n,V_n)) < 1.
\end{equation}
We remind that set $U$ is center-symmetric  if $\thetab \in U$ implies $-\thetab \in U$.

 \begin{thm} \label{tqq} Suppose that set U is bounded, convex and center-symmetric. Then there is consistent tests for some sequence $\rho_n  \to 0$ as $n \to \infty$,  iff, the set $U$ is relatively compact.
 \end{thm}
 If set $U$ is relatively compact, there  is consistent estimator (see \cite{ih}  and \cite{jo}). Therefore we can choose $\mathbb{L}_2$-norm of consistent   estimator as consistent test statistics.

  Similar Theorem holds for signal detection in linear inverse ill-posed problem.

 In Hilbert space $\mathbb{H}$, we observe
  a realization of  Gaussian random vector
 \begin{equation}\label{il1}
  \yb = A\thetab + \epsilon \xib, \quad \epsilon > 0,
  \end{equation}
  where $A: \mathbb{H} \to \mathbb{H}$ is known  linear operator and $\xib$ is Gaussian random vector having known covariance operator $R: \mathbb{H} \to \mathbb{H}$ and
  $\mathbf{E} [\xib] = 0$.

  We explore the same problem of hypothesis testing  $\mathbb{H}_0 : \thetab = 0$ versus alternatives $\mathbb{H}_n\, :\, \thetab \in V_n$.

For any operator $S: \mathbb{H} \to \mathbb{H}$ denote $\frak{R}(S)$ the rangespace of $S$.

  Suppose that the nullspaces of $A$ and $R$ equal  zero and $\frak{R}(A) \subseteq \frak{R}(R^{1/2})$.

\begin{thm}\label{tqq1} Let  operator $R^{-1/2}A$ be bounded. Suppose that set U is bounded, convex and center-symmetric. Then the statement of Theorem \ref{tqq} holds.
\end{thm}
\begin{remark} In definition of consistency  we can replace (\ref{uux})  the requirement of existence of sequence of tests $K_n$ such that $\alpha(K_n) \to 0$ and $\beta(K_n,V_n) \to 0$ as $n \to \infty$. By Theorem on exponential decreasing of type I and type II error probabilities (see  \cite{le73} and  \cite{sch}), Theorems \ref{tqq} and \ref{tqq1} remain valid for this definition as well. \end{remark}

  \section{Quadratic test statistics \label{sec4}}
  \subsection{General setup \label{s4.1}}
We  explore  problem of signal detection in Gaussian white noise (\ref{q1}), (\ref{i30}) discussed in Introduction. The problem is provided in terms of sequence model (\ref{q2}).

 If $U$ is compact ellipsoid in $\mathbb{L}_2(0,1)$,  asymptotically minimax test statistics are  quadratic forms
$$
T_n(Y_n) = \sum_{j=1}^\infty \kappa_{nj}^2 y_j^2 - \sigma^2 n^{-1} \rho_n
$$
with some specially defined coefficients $\kappa^2_{nj}$ (see Ermakov \cite{er90}). Here $\rho_n = \sum_{j=1}^\infty \kappa_{nj}^2$.

If coefficients $\kappa_{nj}^2$ satisfy some regularity assumptions,  test statistics $T_n(Y_n)$ are asymptotically minimax (see \cite{er04}) for the  wider sets of alternatives
$$
\mathbb{H}_n : f \in Q_n(c)  = \{\, f:  R_n(f) > c ,\,\, f \in \mathbb{L}_2(0,1) \,\}
$$
with
$$
R_n(f) = A_n(\thetab) = \sigma^{-4}\,n^{2}\,\sum_{j=1}^\infty\, \kappa_{nj}^2\,\theta_j^2.
$$
for $f = \sum_{j=1}^\infty \theta_j \phi_j$.

A sequence of tests $L_n, \alpha(L_n) = \alpha(1+ o(1))$, $0 <\alpha<1$, is called {\sl asymptotically minimax}  if, for any sequence of tests $K_n, \alpha(K_n) \le \alpha,$ there holds
\begin{equation*}
\liminf_{n\to \infty}(\beta(K_n,Q_n(c)) - \beta(L_n,Q_n(c))) \ge 0.
\end{equation*}
Sequence of test statistics $T_n$ is asymptotically minimax if   tests generated test statistics $T_n$ are asymptotically minimax.

We make the following  assumptions.

\noindent{\bf A1.} For each $n$  sequence $\kappa^2_{nj}$ is decreasing.

\noindent{\bf A2.} There are positive constants $C_1,C_2$ such that, for each $n$, there holds
\begin{equation}\label{q5}
 C_1 < A_n = \sigma^{-4}\,n^2\,\sum_{j=1}^\infty \kappa_{nj}^4 < C_2.
 \end{equation}

 \noindent{\bf A3.} There are positive constants $c_1$ and $c_2$ such that $c_1n^{-2r} \le \rho_n \le c_2 n^{-2r}$.

  Denote $\kappa_n^2=\kappa^2_{nk_n}$ with $k_n = \sup\Bigl\{k: \sum_{j < k} \kappa^2_{nj} \le \frac{1}{2} \rho_n \Bigr\}$.

\noindent{\bf A4.}   There are $C_1$  and $\lambda >1$ such that, for any $\delta > 0$ and for each $n$,
 \begin{equation*}
\kappa^2_{[n,(1+\delta)k_n]} < C_1(1 +\delta)^{-\lambda}\kappa_n^2.
\end{equation*}
\noindent{\bf A5.} There holds $\kappa_{1n}^2  \asymp \kappa_n^2$.  For any $c>1$  there is $C$ such that $\kappa_{[ck_n],n}^2 \ge C\kappa_n^2$ for all $n$.

\noindent{\sl Example}.  Let
$$
\kappa^2_{nj} = n^{-\lambda}\frac{1}{j^{\gamma}  + c n^\beta}, \quad \gamma >1,
$$
with $\lambda = 2 - 2r -\beta$ and $\beta = (2-4r)\gamma$.
Then A1 -- A5 hold.

Note that A1-A5 imply
\begin{equation}\label{u1}
\kappa_n^4=\kappa^4_{nk_n}  \asymp n^{-2}k_n^{-1}\quad \mbox{and} \quad k_n \asymp n^{2-4r}.
\end{equation}
Theorems \ref{tq3} - \ref{tq8} given below represent  realization of program announced in Introduction.
 \subsection{Analytic form of necessary and sufficient conditions of sufficiency \label{s4.2}}
 The results will be provided in terms of  Fourier coefficients of functions $f_n = \sum_{j=1}^\infty \theta_{nj} \phi_j$.
\begin{thm}\label{tq3} Assume {\rm A1-A5}. Sequence of alternatives $f_n$, $cn^{-r} \le \|f_n\| \le Cn^{-r}$, is consistent, iff, there are $c_1$, $c_2$ and $n_0$  such that there holds
\begin{equation}\label{con2}
\sum_{|j| < c_2k_n} |\theta_{nj}|^2 > c_1 n^{-2r}
\end{equation}
for all $n > n_0$.
\end{thm}
Versions of Theorems \ref{tq3}, \ref{tq4} and \ref{tq6} hold for setups of other sections. In these sections indices $j$ may accept negative values and $\theta_{nj}$ may be complex numbers. By this reason we write $|j|$ instead of $j$ and $|\theta_{nj}|$ instead of $\theta_{nj}$ in (\ref{con2}), (\ref{con3}) and (\ref{con19}).
\begin{thm}  \label{tq4}  Assume {\rm A1-A5}.  Sequence of alternatives $f_n$, $cn^{-r} \le \|f_n\| \le Cn^{-r}$, is inconsistent, iff, for all  $c_2$, there holds
\begin{equation}\label{con3}
\sum_{|j| < c_2k_n} |\theta_{nj}|^2 = o( n^{-2r})\quad \mbox{as} \quad {n \to \infty}.
\end{equation}
\end{thm}

Proof of Theorems is based on Theorem \ref{tq2} on asymptotic minimaxity of test statistics $T_n$.

Define sequence of tests $K_n(Y_n) = {\bf 1}_{\{n^{-1}T_n(Y_n) > (2A_n)^{1/2} x_\alpha\}}$, $0 < \alpha <1$, where $x_\alpha$ is defined by the equation $\alpha = 1 - \Phi(x_\alpha)$.
\begin{thm}\label{tq2}
Assume {\rm A1-A5}. Then sequence of tests $K_n(Y_n)$ is asymptotically minimax for the sets  $Q_n(c)$ of alternatives.
There hold $\alpha(K_n) = \alpha + o(1)$ and
\begin{equation}\label{aq1}
\beta(K_n,f_n) = \Phi(x_{\alpha} - R_n(f_n)(2A_n)^{-1/2})(1+o(1))
\end{equation}
uniformly onto all sequences f such that $R_n(f_n)< C$.
\end{thm}
A version of Theorem \ref{tq2} for the problem of signal detection
with  heteroscedastic white noise  has been proved in \cite{er03}.

Such a form of conditions in Theorems \ref{tq3} and \ref{tq4} is caused concentration of coefficients $\kappa_{nj}^2$ in the zone $1  \le j < \infty$ of test statistics $T_n$ and $A_n(\thetab_n)$.
\subsection{Maxisets. Qualitative structure of consistent sequences of alternatives}
Denote $s = \frac{r}{2 -4r}$. Then $r = \frac{2s}{1 + 4s}$.
\begin{thm}\label{tq1} Assume {\rm A1-A5}. Then the balls  $\mathbb{\bar B}^s_{2\infty}(P_0)$ are maxisets for  test statistics $T_n(Y_n)$. \end{thm}
Asymptotically minimax tests  have been found in \cite{er18}  for maxisets $\mathbb{\bar B}^s_{2\infty}(P_0)$ with  deleted "small" $\mathbb{L}_2$- ball  and, in \cite{ing02},   for Besov bodies in $\mathbb{ B}^s_{2\infty}$ defined in terms of  wavelets  coefficients.

Balls $\mathbb{\bar B}^s_{2\infty}(P_0)$ arisee in Theorem \ref{tq1} on the following reason. If $f_n = \sum_{j=1}^\infty \theta_{nj} \phi_j$ satisfies $c_1n^{-r} \le \|f_n\| \le C_1n^{-r}$, then, for any $c$, there is $P_0$ such that $f_{1n} = \sum_{j=1}^{[ck_n]} \theta_{nj} \phi_j \in \mathbb{\bar B}^s_{2\infty}(P_0)$   (see Lemma \ref{ld3}). Since coefficients $\kappa^2_{nj}$ with $j > ck_n$ are small for sufficiently large $c$ this allows to prove Theorem \ref{tq3} and Theorems \ref{tq7}, \ref{tq11} given below.
\begin{thm}\label{tq7} Assume {\rm A1-A5}. Then   sequence of alternatives $f_n$, $c n^{-r}\le \|f_{n}\| \le C n^{-r}
$, is consistent, iff,  there are maxiset $\gamma U$, $\gamma > 0$, $0 < \alpha < 1$, and sequence  $f_{1n} \in \gamma U$, $c_1 n^{-r}\le \|f_{1n}\| \le C_1 n^{-r}
$, such that there holds
\begin{equation}
\label{ma1}
\| f_n \|^2 =  \| f_{1n}\|^2  + \|f_n - f_{1n}\|^2.
\end{equation}
\end{thm}
\begin{thm}\label{tq11} Assume {\rm A1-A5}. Then,  for any $\varepsilon > 0$, for any $\alpha$, $0 < \alpha < 1$, and for any positive constants $c$  and $C$, $c < C$, there are  $\gamma_\varepsilon$ and $n_\varepsilon$ satisfying the following requirement:
\vskip 0.15cm
 \noindent if sequence of alternatives $f_n$,   $c n^{-r}\le \|f_{n}\| \le Cn^{-r}
$, is consistent, then there is sequence of functions $f_{1n}$ belonging to maxiset $ \gamma_\varepsilon U$, $c_1 n^{-r}\le \|f_{1n}\| \le C_1n^{-r}
$, such that (\ref{ma1}) holds and,
for any $n > n_\varepsilon$, there hold
\begin{equation}
\label{uuu}
|\beta(K_n,f_n) - \beta(K_n,f_{1n})| \le \varepsilon
\end{equation}
and
\begin{equation}
\label{uu1}
\beta(K_n,f_n-f_{1n})  \ge 1 - \alpha - \varepsilon.
\end{equation}
Here $K_n$, $\alpha(K_n) = \alpha (1+ o(1))$ as $n \to \infty$, is  sequence of tests generated test statistics $T_n$.
\end{thm}
\subsection{Interaction of consistent and inconsistent sequences. Purely consistent sequences}
\begin{thm}  \label{tq5}  Assume {\rm A1-A5}. Let sequence of alternatives $f_n$ be consistent. Then, for any inconsistent sequence of alternatives $f_{1n}$, for tests $K_n$, $\alpha(K_n) = \alpha (1 + o(1))$, $0 < \alpha <1$, generated test statistics $T_n$, there holds
\begin{equation*}
\lim_{n \to \infty} (\beta(K_n,f_n) - \beta(K_n,f_n + f_{1n})) = 0.
\end{equation*}
\end{thm}
\begin{thm}\label{tq6} Assume {\rm A1-A5}. Sequence of alternatives $f_n$, $cn^{-r} \le \|f_n\| \le Cn^{-r}$, is purely $n^{-r}$-consistent, iff, for any $\varepsilon >0$, there is $C_1= C_1(\varepsilon)$ such that there holds
\begin{equation}\label{con19}
\sum_{|j| > C_1k_n} |\theta_{nj}|^2 \le \varepsilon n^{-2r}
\end{equation}
for all $n> n_0(\varepsilon)$.
\end{thm}
\begin{thm}\label{tq12} Assume {\rm A1-A5}. Then sequence $f_n$, $c n^{-r}\le \|f_{n}\| \le C n^{-r}
$, is  purely $n^{-r}$-consistent, iff, for any $\varepsilon > 0$,  there is $\gamma_\epsilon$ and sequence of functions $f_{1n}$  belonging to maxiset $\gamma_\epsilon U$ such that  $\|f_n - f_{1n}\| \le \varepsilon n^{-r}$ and (\ref{ma1}) holds.
\end{thm}
\begin{thm}\label{tq8} Assume {\rm A1-A5}. Then sequence of alternatives $f_n$, $cn^{-r} < \|f_n\| < Cn^{-r}$, is purely $n^{-r}$-consistent, iff, for any  $n^{-r}$-inconsistent subsequence of alternatives $f_{1n_i}$,  there holds
\begin{equation}\label{ma2}
\|f_{n_i} + f_{1n_i}\|^2 =  \|f_{n_i} \|^2 + \| f_{1n_i}\|^2 + o(n_i^{-r}),
\end{equation}
where $n_i \to \infty$ as  $i \to \infty$.
\end{thm}
\begin{remark}\label{rem1}{\rm Let $\kappa_{nj}^2 > 0$ for $j \le l_n$  and  let $\kappa^2_{nj} = 0$ for $j > l_n$ with $l_n \asymp n^{2-4r}$ as $n \to \infty$.
Analysis of  proofs of Theorems    shows that Theorems \ref{tq3} - \ref{tq8}  remain valid for this setup  if A4 and A5 are replaced with
\vskip 0.25cm
\noindent{\bf A6.} For any $c$, $0 <c <1$, there is $c_1$  such that $\kappa^2_{n,[cl_n]} \ge c_1 \kappa^2_{n1}$ for all $n$.
\vskip 0.25cm
In the reasoning we put $\kappa^2_n = \kappa_{n1}^2$ and $k_n = l_n$.

Theorems   \ref{tq4} and \ref{tq6} hold with the following changes. It suffices to put $c_2 < 1$  in Theorem \ref{tq4} and to take $C_1(\epsilon) < 1$ in Theorem \ref{tq6}.

Proof of corresponding versions of Theorems \ref{tq3} - \ref{tq8} is obtained by simplification of provided reasoning and is omitted.}
\end{remark}
\section{Kernel-based tests \label{sec5}}
We  explore  problem of signal detection of previous section and suppose additionally that functions $f_n$ belong to $\mathbb{L}_2^{per}(\mathbb{R}^1)$ the set of 1-periodic functions such that $f_n(t) \in \mathbb{L}_2(0,1)$. This allows to extend our model on real line $\mathbb{R}^1$ putting $w(t+j) = w(t)$ for all integer $j$ and $t \in [0,1)$ and  to write the forthcoming integrals over all real line.

Define kernel estimator
\begin{equation}\label{yy}
\hat{f}_n(t) = \frac{1}{h_n} \int_{-\infty}^{\infty} K\Bigl(\frac{t-u}{h_n}\Bigr)\, d\,Y_n(u), \quad t \in (0,1),
\end{equation}
where $h_n$ is a sequence of positive numbers, $h_n \to 0$ as $n \to 0$.

The kernel $K$ is bounded function such that the support of $K$ is contained in $[-1,1]$, $K(t) = K(-t)$ for $t \in \mathbb{R}^1$ and $\int_{-\infty}^\infty K(t)\,dt = 1$.

Denote $K_h(t) = \frac{1}{h} K\Bigl(\frac{t}{h}\Bigr)$, $t \in \mathbb{R}^1$ and $h >0$.

In (\ref{yy}) we suppose that, for any $v$, $0 <v  < 1$, we have
\begin{equation*}
 \int_{1}^{1+v} K_{h_n}(t-u)\,dY_n(u) = \int_{0}^{v} K_{h_n}(t-1-u)\,f(u)\,du
 + \frac{\sigma}{\sqrt{n}} \int_{0}^{v} K_{h_n}(t-1-u)\, dw(u)
\end{equation*}
and
\begin{equation*}
 \int_{-v}^{0} K_{h_n}(t-u)\,dY_n(u) =  \int_{1-v}^{1} K_{h_n}(t-u+1)\,f(u)\,du
 + \frac{\sigma}{\sqrt{n} } \int_{1-v}^{1} K_{h_n}(t-u+1)\, dw(u).
\end{equation*}
Define kernel-based   test statistics $$T_n(Y_n) = T_{nh_n}(Y_n) =nh_n^{1/2}\sigma^{-2}\gamma^{-1} (\|\hat f_{n}\|^2- \sigma^2(nh_n)^{-1}\|K\|^2),$$  where
$$
\gamma^2 = 2 \int_{-\infty}^\infty \Bigl(\int_{-\infty}^\infty K(t-s)K(s) ds\Bigr)^2\,dt.
$$
We call sequence of alternatives $f_n$, $cn^{-r} \le \|f_n\| \le Cn^{-r}$, $n^{-r}$-{\sl consistent} if, there is constant $c_1$ such that (\ref{vas1}) holds for any tests $K_n$, $\alpha(K_n) = \alpha\,(1 + o(1))$. $0 < \alpha <1$, generated sequence of test statistics $T_n$ with  $h_n < c_1 n^{4r-2}$, $h_n \asymp n^{4r-2}$.

 We call sequence of alternatives $f_n$, $cn^{-r} \le \|f_n\| \le Cn^{-r}$, $n^{-r}$-{\sl inconsistent} if sequence of alternatives $f_n$ is  inconsistent for all test statistics $T_n$.

Problem will be explored in terms of sequence model.

Let we observe a realization of random process $Y_n(t)$ with $f= f_n$.

 For $-\infty < j < \infty$, denote
$$
\hat K(jh) = \int_{-1}^1 \exp\{2\pi ijt\}\,K_h(t)\, dt,\quad h > 0,
$$
$$
y_{nj} = \int_0^1 \exp\{2\pi ijt\}\, dY_n(t),
\quad
\xi_j = \int_0^1 \exp\{2\pi ijt\}\, dw(t),
$$
$$
\theta_{nj} = \int_0^1 \exp\{2\pi ijt\}\, f_n(t)\, dt.
$$
In this notation we can write kernel estimator in the following form
\begin{equation}\label{au1}
\hat \theta_{nj} = \hat K(jh_n)\, y_{nj} =
\hat K(jh_n)\, \theta_{nj} + \sigma\, n^{-1/2}\, \hat K(jh_n)\, \xi_j, \quad -\infty < j < \infty,
\end{equation}
and test statistics $T_n$ admit the following representation
\begin{equation}\label{au111}
T_n(Y_n) = nh_n^{1/2}\sigma^{-2}\gamma^{-1} \Bigl(\sum_{j=-\infty}^\infty |\hat \theta_{nj}|^2  -  n^{-1}\sigma^2 \sum_{j=-\infty}^\infty |\hat K(jh_n)|^2\Bigr).
\end{equation}
If we put $|\hat K(jh_n)|^2 = \kappa^2_{nj}$, we get that definitions of test statistics $T_n(Y_n)$ in sections   \ref{sec4}. Thus setup of section  \ref{sec5} differs only heteroscedastic white noise. Another difference  in the reasoning is that the function $\hat K(\omega)$, $\omega \in \mathbb{R}^1$, may have zeros.  Since  differences are insignificant the same results are valid. 
Denote $k_n = [n^{2-4r}]$.
\begin{thm}\label{tk3}  The statements of Theorems \ref{tq3}, \ref{tq4}, \ref{tq7}-\ref{tq8} hold for this setup. The statement of Theorem \ref{tq1} holds also with $ \mathbb{\bar B}^s_{2\infty}$  replaced with $\mathbb{B}^s_{2\infty}$.
\end{thm}
In version of Theorem \ref{tq1},  {\sl iv.} in definition of maxisets holds for test statistics $T_n$ having arbitrary values $h_n > 0$, $h_n \to 0$ as $n \to \infty$.

\section{ $\chi^2$-tests \label{sec6}}
Let $X_1,\ldots,X_n$ be i.i.d.r.v.'s having c.d.f. $F(x)$, $x \in (0,1)$.
Let c.d.f. $F(x)$ have a density $1 + f(x) = dF(x)/dx$, $x \in (0,1)$, $f \in L_2^{per}(0,1)$.

We explore the problem of testing hypothesis (\ref{i1}) versus alternatives (\ref{i30}) discussed in Introduction.

Denote $\hat F_n(x)$  empirical c.d.f. of $X_1,\ldots,X_n$.

For any sequence $m_n$, denote  $\hat p_{nj} = \hat F_n(j/m_n) - \hat F_n((j-1)/m_n)$, $1 \le j \le m_n$.

Test statistics of $\chi^2$-tests equal
$$
T_n(\hat F_n) = n\, m_n \, \sum_{j=1}^{m_n}\, (\hat p_{nj}  - 1/m_n)^2.
$$
Let
$$ f_n = \sum_{j=-\infty}^\infty \theta_{nj} \phi_j, \quad \phi_j(x) = \exp\{2\pi i\,j\,x\,\}, \quad x \in (0,1).
$$
We call sequence of alternatives $f_n$, $cn^{-r} \le \|f_n\| \le Cn^{-r}$, $n^{-r}$-{\sl consistent}, if there is $c_1$ such that, (\ref{vas1}) holds for any tests $K_n$, $\alpha(K_n) = \alpha\,(1 + o(1))$. $0 < \alpha <1$, generated sequence of chi-squared test statistics $T_n$ with number of cells $m_n > c_1 n^{2-4r}$, $m_n \asymp n^{2-4r}$.

We call sequence of alternatives $f_n$, $cn^{-r} \le \|f_n\| \le Cn^{-r}$, $n^{-r}$-{\sl inconsistent} if sequence of alternatives $f_n$ is  inconsistent for all tests generated arbitrary test statistics $T_n$.

Denote $k_n = \Bigl[n^\frac{2}{1 + 4s}\Bigr]\asymp n^{2-4r}$.

The differences in versions of Theorems \ref{tq3} --\ref{tq8} for this setup are caused only the requirement that functions $f_n, f_{1n}, f_{2n}$ should be densities.
\begin{thm}\label{tchi3}  The statements of Theorems \ref{tq3}, \ref{tq4} and \ref{tq1}-\ref{tq8} hold for this setup with the following additional assumptions.

 In version of Theorem \ref{tq1} balls $ \mathbb{\bar B}^s_{2\infty}$  is replaced with bodies $\mathbb{\tilde B}^s_{2\infty}$.

 In version  of  Theorem \ref{tq1},   {\sl iv.} in definition of maxisets holds for test statistics $T_n$ with arbitrary choice of number of cells $m_n$.

In version of Theorem \ref{tq11}  we consider only sequences of alternatives $f_n$ such that the following assumption holds.

{\bf B.} There is $c_0$ such that, for all $c > c_0$, functions
$$
1 + f_{cn} =  1 + \sum_{|j| > c m_n} \theta_j \phi_j \quad \mbox{and} \quad 1 + f_n - f_{cn} =  1 + \sum_{|j| < c m_n} \theta_j \phi_j
$$
are densities.

The statement  of Theorem \ref{tq5}  holds only if functions $1 + f_n + f_{1n}$ are densities

We implement definition of purely consistent sequences only for sequences $f_n$ satisfying {\rm B}.
\end{thm}
In proof of version of Theorem  \ref{tq11} for chi-squared tests, we show that there is $C_\varepsilon = C(\varepsilon,c,C,c_0)$ such that, for densities $1 + f_{1n} =  1 + \sum_{|j| < C_\varepsilon m_n} \theta_j \phi_j$, (\ref{ma1}), (\ref{uuu}) and (\ref{con19}) hold. By Lemma \ref{ld3} given below, there is $\gamma_\varepsilon$ such that $f_{1n} \in \gamma_\varepsilon U$.

Proof of Theorems are based on the following Theorem \ref{chi2} on asymptotic minimaxity of chi-squared tests given below. Theorem \ref{chi2} is  summary of results of Theorems 2.1 and 2.4 in \cite{er97}.

For c.d.f. $F$, denote  $p_{l} =  F(l/m_n) -  F((l-1)/m_n)$, $1 \le l \le m_n$.

Denote $\Im$ the set of all distribution functions.

Define  functionals
$$
T_n(F)= n m_n  \sum_{l=1}^{m_n} (p_{l}  - 1/m_n)^2.
$$
For sequence $\rho_n > 0$, define  sets of alternatives
$$
Q_n(\rho_n) = \Bigl\{\, F: T_n(F) \ge \rho_n,\, F \in \Im\, \Bigr\}.
$$
The definition of asymptotic minimaxity of tests is the same as in section \ref{sec4}.

Define the tests
$$
K_n = {\bf 1}_{\{2^{-1/2}m_n^{-1/2}(T_n(\hat F_n) - m_n + 1) > x_\alpha\}}
$$
where $x_\alpha$ is defined the equation $\alpha = 1- \Phi(x_\alpha)$.
\begin{thm}\label{chi2} Let $m_n \to \infty$, $m_n^{-1} n^2 \to \infty$ as $n \to \infty$. Let
\begin{equation*}
0 <\, \liminf_{n \to \infty}\, m_n^{-1/2}\rho_n \le \limsup_{n \to \infty}\, m_n^{-1/2} \rho_n < \, \infty.
\end{equation*}
Then $\chi^2$-tests  $K_n$, $\alpha(K_n) = \alpha + o(1)$, $0 < \alpha <1$, are asymptotically minimax for the sets of alternatives $Q_n(\rho_n)$.
There holds
\begin{equation*}
\beta(K_n,F) = \Phi(x_\alpha - 2^{-1/2}m_n^{-1/2}T_n(F_n))(1 +o(1))
\end{equation*}
uniformly onto sequences $F_n$ such that $ T_n(F_n) \le C m_n^{1/2} $.
\end{thm}
Note that for implementation of Theorem \ref{chi2} we need to make a transition from indicator functions to trygonometric functions. Such a transition is realised in Appendix.
\section{ Cramer -- von Mises tests \label{sec7}}
We  consider Cramer -- von Mises test statistics as functional
$$
T^2(\hat{F}_n -F_0) = \int_0^1 (\hat{F}_n(x) - F_0(x))^2\, d F_0(x)
$$
depending on empirical distribution function $\hat F_n$.
Here $F_0(x)=x$, $x \in (0,1)$.

Denote $K_n$ sequence of Cramer- von Mises tests.

 A part of further results holds if we consider as alternatives sequence of c.d.f.'s $F_n$ instead of sequence of densities $1 + f_n$.   We shall suppose that c.d.f.'s $F_n$ are Borel functions. Denote $\beta_F(K_n)$ - type II error probability for alternative $F$.

 For any $a>0$, denote $\Im_n(a) = \{F\, : \, n T^2(F - F_0)  > a,\, F \mbox{ is c.d.f.}\}$.

We say that Cramer - von Mises test is asymptotically unbiased if, for any $a > 0$, for any $\alpha$, $0 < \alpha < 1$, for tests $K_n$, $\alpha(K_n) = \alpha + o(1)$, there holds
\begin{equation}\label{dur}
 \limsup_{n \to \infty}\sup_{F \in \Im_n(a)} \beta_F(K_n) < 1 -\alpha.
\end{equation}
 Nonparametric tests satisfying (\ref{dur}) are called also uniformly consistent (see Ch. 14.2 in \cite{le}).

 The results are based on the following Theorem \ref{tcm}.

 \begin{thm}\label{tcm} The following three statements hold.

 {\sl i.} For sequence of alternatives $F_n$, there is sequence of Cramer - von Mises tests $K_n$
 such that
 \begin{equation}\label{cru0}
 \lim_{n \to \infty} (\alpha(K_n) + \beta_{F_n}(K_n)) = 0,
 \end{equation}
 holds, iff, there holds
 \begin{equation} \label{cru}
 \lim_{n \to \infty} n\, T^2(F_n - F_0) = \infty.
 \end{equation}

 {\sl ii.} Cramer - von Mises tests are asymptotically unbiased.

 {\sl iii.} For any sequence of Cramer - von Mises tests $K_n$,
 \begin{equation*}
 \lim_{n \to \infty} (\alpha(K_n) + \beta_{F_n}(K_n)) \ge 1,
 \end{equation*}
 holds, iff, there holds
 \begin{equation*}
 \lim_{n \to \infty} n\, T^2(F_n - F_0) = 0.
 \end{equation*}
 \end{thm}
 Sufficiency in {\sl i.} and {\sl iii.} in Theorem \ref{tcm} is wellknown (see  \cite{ing87}).
Necessary conditions in {\sl i.} and in {\sl iii.}  follows  easily  from {\sl ii.}

 If c.d.f. $F$ has density, we can write the functional $T^2(F-F_0)$ in the following form (see Ch.5,  \cite{wel})
\begin{equation*}
T^2(F-F_0) = \int_0^1\int_0^1(\min\{s,t\} - st)\, f(t)\, f(s)\, ds\, dt
\end{equation*}
with $f(t) = d(F(t) - F_0(t))/dt$.

If we consider the orthonormal expansion of function
\begin{equation*}
f(t) =  \sum_{j=1}^\infty \theta_j \phi_j(t)
\end{equation*}
on trigonometric basis $\phi_j(t) = \sqrt{2}\, \cos(\pi j t)$, $1 \le j <\infty$, then we get
\begin{equation}\label{om3}
nT^2(F-F_0) = n\,\sum_{j=1}^\infty\frac{\theta_j^2}{\pi^2 j^2}.
\end{equation}
Denote $k_n = [n^{(1-2r)/2}]$.

In Theorems \ref{tom1}, \ref{tom6} and \ref{tom4} given below, we follow the definition of consistency provided in subsection \ref{ss2.1}.
 \begin{thm}\label{tom1} For orthonormal system of functions $\phi_j(t) = \sqrt{2}  \cos(\pi j t)$, $t \in [0,1)$, $j = 1,2, \ldots$,
the bodies  $\mathbb{\bar B}^s_{2\infty}(P_0)$ with $s = \frac{2r}{1 - 2r}$,  $r = \frac{s}{2+2s}$, are maxisets for Cramer -- von Mises test statistics.
\end{thm}
\begin{thm}\label{tom6} The statement of Theorem \ref{tq11} holds for this setup with the following difference: we  assume that B holds.
\end{thm}
In Theorem \ref{tom5} given below we consider the definition of consistency proposed in subsection \ref{ss2.1} for c.d.f'.s $F_n$ instead of sequence $f_n$.
\begin{thm}  \label{tom5}   Let sequence of alternatives $F_n$  be consistent. Let $F_{1n}$ be any inconsistent sequence of alternatives such that $F_{2n} = F_n(x) + F_{1n}(x) - F_0(x)$  are c.d.f.'s. Then, for tests $K_n$, $\alpha(K_n) = \alpha (1 + o(1))$, $0 < \alpha < 1$,  there holds
\begin{equation*}
\lim_{n \to \infty} (\beta_{F_n}(K_n) - \beta_{F_{2n}}(K_n)) = 0.
\end{equation*}
\end{thm}

In previous sections functionals $T_n$ depend on $n$. In this setup we explore the unique functional $T$ for all $n$ and different values of $r$, $0 < r < 1/2$. To separate the study of sequences of alternatives for different $r$, we consider only sequences of alternatives satisfying G1.

{\bf G1.} For any  $\varepsilon > 0$   there is $c_3$ such that there holds
\begin{equation*}
n\sum_{|j| < c_3 k_n} \theta_{nj}^2 j^{-2} < \varepsilon
\end{equation*}
for all $n  > n_0(\varepsilon,c_3)$.

If  G1 does not hold with any  $c_n \to 0$, $c_n k_n \to \infty$ as $n \to \infty$ and functions $ 1 + \bar f_n = 1 + \sum_{j < c_n k_n} \theta_{nj}\,\phi_j$ are densities, then  (\ref{vas1}) holds for some sequence of functions $\bar f_n$, $\|\, \bar f_n \| = o(n^{-r})$. Thus this case of consistency can be studied in the framework of the faster rate of convergence of sequence of alternatives.

\begin{thm}  \label{tom4} Let sequence of alternatives $f_n$ satisfies \mbox{\rm G1}. Then for sequence $f_n$ the statements of Theorems \ref{tq3}, \ref{tq4}, \ref{tq7},  \ref{tq6}, \ref{tq12} and \ref{tq8}  are valid with the following additional assumptions.

In version of Theorem \ref{tq11}, $\gamma_\epsilon$ and $n_\epsilon$ depend on sequence $f_n$, $cn^{-r} < \|f_n\| < Cn^{-r}$.

 In Theorem \ref{tom4} definition of pure consistency is considered for sequences of functions $f_n$ satisfying \mbox{\rm B}.
\end{thm}
\appendix

\section{Proof of Theorems}\label{app}

 \subsection{Proof of Theorem \ref{tqq} \label{subsec9.3}}
 For any vectors $\thetab_1 \in \mathbb{H}$ and $\thetab_2 \in \mathbb{H}$ define segment $\frak{int}(\thetab_1,\thetab_2) = \{\thetab : \thetab = (1 - \lambda)\,\thetab_1 + \lambda \,\thetab_2, \,\, \lambda \in [0,1]\,\}$.

 Proof of Theorem \ref{tqq} is based on the following Lemma \ref{lqq1}.
 \begin{lemma}\label{lqq1} For any vectors $\thetab_1 \in U$ and $\thetab_2 \in U$ we have $\frak{int}\Bigl(\frac{\thetab_1-\thetab_2}{2},\frac{\thetab_2-\thetab_1}{2}\Bigr) \subset U$.
 There  holds $0 \in \frak{int}\Bigl(\frac{\thetab_1-\thetab_2}{2},\frac{\thetab_2-\thetab_1}{2}\Bigr)$ and segment $ \frak{int}\Bigl(\frac{\thetab_1-\thetab_2}{2},\frac{\thetab_2-\thetab_1}{2}\Bigr)$ is parallel to segment  $\frak{int}(\thetab_1,\thetab_2)$.
 \end{lemma}

 {\sl Remark 3.1.} Let  we have segment $\frak{int}(\thetab_1,\thetab_2) \subset U$.
 Let $\etab$ and $-\etab$ be the points of intersection of the line $L = \{\thetab : \thetab = \lambda(\thetab_1 - \thetab_2), \lambda \in \mathbb{R}^1\}$ and the boundary of set $U$. Then, by Lemma \ref{lqq1}, we have $\|\thetab_1 - \thetab_2\| \le 2\|\etab\|$.

 \begin{proof}[ Proof of Lemma \ref{lqq1}]. Segments $\frak{int}(\thetab_1,\thetab_2) \subset U$ and $\frak{int}(-\thetab_1,-\thetab_2) \subset U$ are parallel.  For each $\lambda \in [0,1]$ we have $ (1-\lambda)\thetab_1 + \lambda \thetab_2 \in \frak{int}(\thetab_1,\thetab_2)$ and $-\lambda \thetab_1 - (1 - \lambda)\thetab_2 \in \frak{int}(-\thetab_1,-\thetab_2)$.
 The middle $\thetab_\lambda= ((1-2\lambda)\,\thetab_1 - (1-2\lambda)\,\thetab_2)/2$ of segment  $\frak{int}((1-\lambda)\thetab_1 + \lambda \thetab_2, -\lambda \thetab_1 - (1 - \lambda)\thetab_2) \subset U$ belongs to segment $\frak{int}\Bigl(\frac{\thetab_1-\thetab_2}{2},\frac{\thetab_2-\thetab_1}{2}\Bigr)$ and, for each point $\thetab$ of segment  $\frak{int}\Bigl(\frac{\thetab_1-\thetab_2}{2},\frac{\thetab_2-\thetab_1}{2}\Bigr)$, there is $\lambda\in [0,1]$ such that $\thetab=\thetab_\lambda$ . Therefore $\frak{int}\Bigl(\frac{\thetab_1-\thetab_2}{2},\frac{\thetab_2-\thetab_1}{2}\Bigr) \subset U$.
 \end{proof}

 \begin{proof}[ Proof of Theorem \ref{tqq}]. Without loss of generality we can suppose that the set $U$ is closed.
   Define sequence of orthogonal vectors $\eb_i$ by induction.

  Let  $\eb_1$, $\eb_1 \in U$,  be such that $\|\eb_1\| = \sup\{ \|\thetab\|, \thetab \in U \}$. Denote $\Pi_1$ linear subspace generated $\eb_1$. Denote $\Gamma_1$ linear subspace orthogonal to $\Pi_1$.

  Let $\eb_i \in U \cap \Gamma_{i-1}$ be such that $\| \eb_i \|  = \sup\{ \| \thetab\| : \thetab \in U \cap \Gamma_{i-1}\}$.  Denote $\Pi_i$ linear subspace generated vectors $\eb_1, \ldots, \eb_i$. Denote $\Gamma_i$ linear subspace orthogonal to $\Pi_i$.

 For all natural $i$ denote $d_i = \| \eb_i \|$. Note that $d_i \to 0$ as $i \to \infty$. Otherwise, by Theorem 5.3 in  \cite{er15}, there does not exist consistent test for the problem of testing hypothesis $\mathbb{H}_0 : \thetab = 0$ versus alternative $\mathbb{H}_1 : \thetab = \eb_i$, $i=1,2, \ldots$.

 For any $\varepsilon \in (0,1)$  denote  $l_\varepsilon = \min \{j\,:\, d_j < \varepsilon, j =1,2, \ldots \}$.

Denote $B_r(\thetab)$ ball having radius $r$ and center $\thetab$.

It suffices to show that, for any $\varepsilon_1 > 0$, there is finite coverage of set $U$ by balls $B_{\varepsilon_1}(\thetab)$.

Denote $\varepsilon = \varepsilon_1/9$.

Denote $U_\varepsilon$ projection of set $U$ onto subspace $\Pi_{l_\varepsilon}$.

Denote $\tilde B_r(\thetab)$ ball in $\Pi_{l_\varepsilon}$ having radius $r$ and center $\thetab \in \Pi_{l_\varepsilon}$. There is ball $\tilde B_{\delta_1}(0)$ such that $\tilde B_{\delta_1}(0) \subset U$. Denote $\delta = \min\{\varepsilon,\delta_1\}$.

Let $\thetab_1, \ldots, \thetab_k$ be $\delta$-net in $U_\varepsilon$.

Let $\etab_1, \ldots, \etab_k$ be points of $U$ such that $\thetab_i$ is projection of $\etab_i$ onto subspace $\Pi_{l_\varepsilon}$ for $1 \le i \le k$.

Let us show that $B_{\varepsilon_1}(\etab_1), \ldots, B_{\varepsilon_1}(\etab_k)$  is coverage of set $U$.

Let $\etab \in U$  and let $\thetab$ be projection of $\etab$ onto $\Pi_{l_\varepsilon}$. There is $i$, $1 \le i \le k$, such that  $\|\thetab_i - \thetab\| \le \delta$. It suffices to show that $\etab \in B_{\varepsilon_1}(\etab_i)$.

By Lemma \ref{lqq1}, $\frak{int}\Bigl(\frac{\etab_i-\etab}{2},\frac{\etab -\etab_i}{2}\Bigr) \subset U$. Since $\thetab_i - \thetab \in \Pi_{l_\epsilon}$ and $\thetab_i - \thetab \in \tilde B_{\delta}(0)$, then $(\thetab_i - \thetab)/2 \in U$.  Since $U$ is center-symmetric and convex we  have  $\frac{1}{2}((\etab_i- \etab)/2) - \frac{1}{2}((\thetab_i - \thetab)/2) \in U$. Note that vector $(\etab_i- \thetab_i) - (\etab - \thetab)$ is orthogonal to the subspace $\Pi_{l_\varepsilon}$. Therefore $\|((\etab_i- \thetab_i) - (\etab - \thetab))/4\| \le 2\varepsilon$. Therefore $\|\etab - \etab_i\| \le 8 \varepsilon + \|\thetab - \thetab_i\| < 9\varepsilon$. This implies $\etab \in B_{\varepsilon_1}(\etab_i)$. 
\end{proof}

\begin{proof}[ Proof of Theorem \ref{tqq1}]. Proof of Theorem \ref{tqq} is based on Theorem 5.3 in \cite{er15}.  For linear inverse ill-posed problems (\ref{il1}), Theorem 5.5 in \cite{er15} is akin to Theorem 5.3 in \cite{er15}. Thus it suffices to implement Theorem 5.5 in \cite{er15} instead of Theorem 5.3 in \cite{er15} in  proof of Theorem \ref{tqq}.
\end{proof}
 \subsection{Proof of Theorems of section \ref{sec4} \label{subsec9.4}}
\begin{proof}[ Proof of Theorem \ref{tq2}] Theorem \ref{tq2} and its version for Remark \ref{rem1} setup can be deduced straightforwardly from Theorem 1 in  \cite{er90}.

The lower bound follows from  reasoning of Theorem 1 in \cite{er90} straightforwardly.

The upper bound follows from the following reasoning.  We have
\begin{equation}\label{aq2}
\begin{split}&
 \sum_{j=1}^\infty \kappa_{nj}^2 y_j^2  = \sum_{j=1}^\infty \kappa_{nj}^2 \theta_{nj}^2 + 2\frac{\sigma}{\sqrt{n}}\sum_{j=1}^\infty \kappa_{nj}^2 \theta_{nj} \xi_j +
\frac{\sigma^2}{n} \sum_{j=1}^\infty \kappa_{nj}^2 \xi_j^2\\&
= n^{-2} A_n(\thetab_n) + 2\,J_{1n} + J_{2n}
\end{split}
\end{equation}
with
\begin{equation}\label{aq3}
\mathbf{E} [J_{2n}] = \frac{\sigma^2}{n} \rho_n,
\quad
\mathbf{Var} [J_{2n}] = 2 \frac{\sigma^4}{n^4} A_n
\end{equation}
and
\begin{equation}\label{aq5}
\mathbf{Var} [J_{1n}] = \frac{\sigma^2}{n}  \sum_{j=1}^\infty \kappa_{nj}^4 \theta_{nj}^2 \le \frac{\sigma^2\kappa^2_n}{n}  \sum_{j=1}^\infty \kappa_{nj}^2 \theta_{nj}^2= o(n^{-4}A_n(\thetab_n)).
\end{equation}
By Chebyshov inequality, it follows from (\ref{aq2}) - (\ref{aq5}), that, if $A_n  = o(A_n(\thetab_n))$ as $ n \to \infty$, then $\beta(L_n,f_n) \to 0$ as $n \to \infty$. Thus it suffices to explore the case
\begin{equation}\label{aq6}
A_n\asymp A_n(\thetab_n) = n^2\sum_{j=1}^\infty \kappa_{nj}^2 \theta_{nj}^2.
\end{equation}
If (\ref{aq6}) holds, then, implementing the reasoning of proof of Lemma 1 in \cite{er90}, we get that (\ref{aq1}) holds. \end{proof}
\begin{proof}[ Proof of Theorem  \ref{tq3}]  Let (\ref{con2}) hold. Then, by A5 and (\ref{u1}), we have
\begin{equation*}
A_n(\thetab_n) = n^2 \sum_{j=1}^\infty \kappa_{nj}^2 \theta_{nj}^2 \ge C n^2\kappa^2_n \sum_{j=1}^{c_2k_n}  \theta_{nj}^2 \asymp n^2 \kappa_n^2 n^{-2r}  \asymp 1.
\end{equation*}
By Theorem \ref{tq2}, this implies sufficiency.

Necessary conditions follows from sufficiency conditions in Theorem \ref{tq4}. \end{proof}
\begin{proof}[ Proof of Theorem  \ref{tq4}] Let (\ref{con3}) hold. Then, by (\ref{u1}) and A2, we have
\begin{equation}\label{eq2}
A_n(\thetab_n) \le C n^2\kappa^2_n \sum_{j < c_2k_n}  \theta_{nj}^2 + C n^2 \kappa^2_{n,[c_2n]} \sum_{j > c_2 n} \theta_{nj}^2 \asymp o(1) +  O(\kappa^2_{n,[c_2n]}/\kappa^2_n).
\end{equation}

 By A4, we have
\begin{equation}\label{eq102}
\lim_{c_2 \to \infty}\lim_{n \to \infty} \kappa^2_{n,[c_2n]}/\kappa^2_n \to 0,
\end{equation}

By Theorem \ref{tq2}, (\ref{eq2}) and (\ref{eq102}) together, we get sufficiency.

Necessary conditions follows from sufficiency conditions in Theorem \ref{tq3}. \end{proof}

\begin{proof}[ Proof of Theorem  \ref{tq1}. iii] The statement follows from  Theorem \ref{tq3} and Lemma \ref{ld1} provided below.

\begin{lemma} \label{ld1} Let $f_n \in c_1 U$ and $cn^{-r}\le \|f_n\| \le Cn^{-r}
$. Then, for $l_n = C_1 n^{2 - 4r}(1  +o(1)) = C_1 n^{\frac{r}{s}}(1+o(1)) $ with $C_1^{2s} > 2c_1/c$, there holds
\begin{equation}\label{d1}
\sum_{j=1}^{l_n} \theta_{nj}^2 > \frac{c}{2} n^{-2r}(1 + o(1)).
\end{equation}
\end{lemma}
Proof. Let $f_n \in c_1 U$. Then we have
\begin{equation*}
l_n^{2s} \sum_{j=l_n}^\infty \theta_{nj}^2 = C_1^{2s} n^{2r}  \sum_{j=l_n}^\infty \theta_{nj}^2(1 + o(1))\le c_1(1 + o(1)).
\end{equation*}
Hence
\begin{equation} \label{ucc1}
\sum_{j=l_n}^\infty \theta_{nj}^2 \le c_1 C_1^{-2s} n^{-2r} \le \frac{c}{2} n^{-2r}(1 + o(1)).
\end{equation}
Therefore (\ref{d1}) holds. \end{proof}

\begin{proof}[ Proof of Theorem  \ref{tq1}. iv]  Suppose  opposite. Then  $f = \sum_{j=1}^\infty \tau_{j}\,\phi_j  \notin \mathbb{\bar B}^s_{2\infty}$. This implies that there is sequence $m_l$, $m_l \to \infty$ as $l \to \infty$, such that
\begin{equation}\label{u5}
m_l^{2s} \sum_{j=m_l}^\infty \tau_{j}^2 = C_l
\end{equation}
with $C_l \to \infty$ as $l \to \infty$.

Define a sequence $\etab_l = \{\eta_{lj}\}_{j=1}^\infty$ such that
$\eta_{lj} = 0$ if $j  < m_l$ and $\eta_{lj} = \tau_j$ if $j \ge m_l$.

Since $V$ is convex and ortho-symmetric we have  $\tilde f_l=  \sum_{j=1}^\infty \eta_{lj}\,\phi_j \in V$.

For alternatives $\tilde f_l$ we define sequence $ n_l$ such that
\begin{equation}\label{u5b}
\|\etab_l\|^2 \asymp n_l^{-2r}\asymp m_l^{-2s} C_l .
\end{equation}
Then
\begin{equation}\label{u7}
n_l \asymp C_l^{-1/(2r)} m_l^{s/r} = C_l^{-1/(2r)} m_l^{\frac{1}{2 - 4r}}.
\end{equation}
Therefore we get
\begin{equation}\label{u10}
m_l \asymp C_l^{(1-2r)/r}n_l^{2-4r}.
\end{equation}
By A4, (\ref{u10}) implies
\begin{equation}\label{u9}
\kappa^2_{n_l m_l} = o(\kappa^2_{n_l}).
\end{equation}
Using (\ref{u1}), A2 and (\ref{u9}), we get
\begin{equation}\label{u12}
\begin{split}&
A_{n_l}(\etab_l) = n_l^2 \sum_{j=1}^{\infty} \kappa_{n_lj}^2\eta_{jl}^2 \le n_l^{2} \kappa^2_{m_l n_l}\sum_{j=m_l}^{\infty} \theta_{n_lj}^2\\& \asymp n_l^{2-2r}\kappa^2_{n_l m_l} = O(\kappa^2_{n_l m_l }\kappa^{-2}_{n_l}) = o(1) .
\end{split}
\end{equation}
By  Theorem \ref{tq2},  (\ref{u12}) implies $n_l^{-r}$-inconsistency of   sequence of alternatives $\tilde f_l$. \end{proof}

\begin{proof}[ Proof of Theorem \ref{tq7}]  Theorem \ref{tq7} follows  from Lemmas \ref{ld3} -- \ref{ld6}.
\begin{lemma}\label{ld3} For any  $c$ and any $C$ there is $\gamma$ such that, if $
f_n = \sum_{j =1}^{ck_n} \theta_{nj} \phi_j,
$ and $\| f_n \| \le C n^{-r}$,
then  $f_n \in \gamma U$.
\end{lemma}
\begin{proof} Let $C_1$ be such that $k_n = C_1 n^{r/s}(1 + o(1))$. Then we have
\begin{equation*}
k_n^{2s} \sum_{j=1}^{ck_n} \theta_{nj}^2 \le  C_1 n^{2r}  \sum_{j=1}^{\infty} \theta_{nj}^2  (1 + o(1)) < C C_1 (1 + o(1)).
\end{equation*}
\end{proof}
\begin{lemma}\label{ld4} Let  (\ref{ma1}) hold. Then sequence $f_n$ is $n^{-r}$-consistent.
\end{lemma}
\begin{proof} Let $f_n = \sum_{j=1}^\infty \theta_{nj}  \phi_j$ and let   $f_{1n} = \sum_{j=1}^\infty \eta_{nj}  \phi_j$. Denote $  \zeta_{nj}  = \theta_{nj} - \eta_{nj}$, $1 \le j < \infty$.

For any $\delta > 0$,  $c_1$ and $C_2$, there is $c_2$ such that, for each $f_{1n} \in c_1 \,U$, $\| f_{1n} \| \le C_2 n^{-r}$, there holds
\begin{equation}\label{uh11}
\sum_{j>c_2 k_n} \eta_{nj}^2 < \delta n^{-2r}.
\end{equation}
To prove (\ref{uh11}) it suffices to put $c_2 k_n = l_n = C_1 n^{2-4r}(1 + o(1))$ in (\ref{ucc1}) with $C_1^{2s} > \delta c_1$.

We have
\begin{equation}\label{dub2}
\begin{split}&
J_n=\left| \sum_{j > ck_n} \theta_{nj}^2 - \sum_{j > ck_n} \zeta_{nj}^2\right| \le  \sum_{j > ck_n} |\eta_{nj}(2\theta_{nj} - \eta_{jn})|\\& \le \left( \sum_{j > ck_n}\eta_{nj}^2\right)^{1/2}\left(2\left( \sum_{j > ck_n} \theta_{nj}^2\right)^{1/2} + \left( \sum_{j > ck_n}\eta_{nj}^2\right)^{1/2}\right)
\le  C\delta^{1/2}n^{-2r}.
\end{split}
\end{equation}
By (\ref{ma1}), using (\ref{uh11}) and (\ref{dub2}), we get
\begin{equation}\label{dub3}
\begin{split}&
 \sum_{j < ck_n} \theta_{nj}^2=
 \sum_{j=1}^\infty \eta_{nj}^2  + \sum_{j=1}^\infty \zeta_{nj}^2 - \sum_{j \ge ck_n} \theta_{nj}^2 \ge
  \sum_{j < ck_n} \eta_{nj}^2  -J_n\\&
 \ge  \sum_{j < ck_n} \eta_{nj}^2   - C\delta^{1/2}n^{-2r}
 \ge \|f_{1n}\|^2 - \delta n^{-2r} - C\delta^{1/2}n^{-2r}
 .
 \end{split}
\end{equation}
By Theorem \ref{tq3}, (\ref{dub3}) implies consistency of sequence $f_n$. \end{proof}

\begin{lemma}\label{ld6} Let sequence $f_{n}$, $cn^{-r}\le \|f_{n}\| \le Cn^{-r}
$, be consistent. Then (\ref{ma1}) holds.
\end{lemma}
\begin{proof} By Theorem \ref{tq3}, there  are $c_1$ and $c_2$  such that sequence $f_{1n} = \sum_{j<c_2 k_n} \theta_{nj}  \phi_j$ is consistent and $\|f_{1n}\| \ge c_1 n^{-r}$.
By Lemma \ref{ld3}, there is $\gamma >0$ such that $f_{1n} \in \gamma U$.\end{proof} \end{proof}

\begin{proof}[ Proof of Theorem \ref{tq11}] By A4 and (\ref{u1}), for any $\delta > 0$, there is $c$ such that
\begin{equation}\label{dub29}
n^2 \sum_{j > ck_n} \kappa_{nj}^2 \theta_{nj}^2 \le \delta.
\end{equation}
By Lemma \ref{ld3}, there is $\gamma >0$ such that $f_{1n} = \sum_{j <ck_n} \theta_{nj} \phi_j \in \gamma  U$. By Theorem \ref{tq2} and (\ref{dub29}), for sequence of alternatives $f_{1n}$,
(\ref{uuu}) and (\ref{uu1}) hold. \end{proof}
\begin{proof}[ Proof of Theorem \ref{tq5}].  Let $f_n = \sum_{j=1}^\infty \theta_{nj} \phi_j$  and let  $f_{1n} = \sum_{j=1}^\infty \eta_{nj} \phi_j$. Denote $\etab_n = \{\eta_{nj}\}_{j=1}^\infty$.

By Cauchy inequality, we have
\begin{equation} \label{co1}
\begin{split}&
|A_n(\thetab_n) - A_n(\thetab_n + \etab_n)| = n^2\Bigl| \sum_{j=1}^\infty \kappa_{nj}^2 \theta_{nj}^2  - \sum_{j=1}^\infty \kappa_{nj}^2 (\theta_{nj} + \eta_{nj})^2\Bigr|\\& \le
2\,A_n^{1/2}(\thetab_n)A_n^{1/2}(\etab_n) + A_n(\etab_n) .
\end{split}
\end{equation}
By inconsistency of sequence $f_{1n}$ and Theorem \ref{tq2}, we get $A_n(\etab_n) = o(1)$ as $n \to \infty$. Therefore, by (\ref{co1}),  $|A_n(\thetab_n) - A_n(\thetab_n + \etab_n)| = o(1)$ as $n \to \infty$. Hence, by Theorem \ref{tq2}, we get Theorem \ref{tq5}. \end{proof}

\begin{proof}[ Proof of Theorem \ref{tq6}. Sufficiency] Suppose  opposite. Then there is $n_i \to \infty$ as $i \to \infty$ such that $f_{n_i} = f_{1n_i} + f_{2n_i}$,
\begin{equation}\label{dub401}\|f_{n_i}\|^2 = \|f_{1n_i}\|^2 + \|f_{2n_i}\|^2,
\end{equation}
 $c_1 n_i^{-r}<\|f_{1n_i}\| < C_1 n_i^{-r}$, $c_2 n_i^{-r}<\|f_{2n_i}\| < C_2 n_i^{-r}$ and sequence $f_{2n_i}$ is inconsistent.
 \vskip 0.1cm
 Let $f_{n_i} = \sum_{j=1}^\infty \theta_{n_ij}\phi_j$, $f_{1n_i} = \sum_{j=1}^\infty \theta_{1n_ij}\phi_j$ and $f_{2n_i} = \sum_{j=1}^\infty \theta_{2n_ij}\phi_j$.

  Then, by (\ref{con19}) and by Theorem  \ref{tq4},  we get that there are $\varepsilon_i \to 0$ and $C_i = C(\varepsilon_i) \to \infty$ as $i \to \infty$ such that
\begin{equation}\label{dub301}
\sum_{j > C_ik_n} \theta_{n_ij}^2 =  \sum_{j > C_ik_n} (\theta_{1n_ij} + \theta_{2n_ij})^2 = o(n^{-2r}), \quad
 \sum_{j < C_ik_n} \theta_{2n_ij}^2  = o(n^{-2r}).
\end{equation}
By (\ref{dub401}) and (\ref{dub301}), we get
\begin{equation}\label{dub402}
\sum_{j = 1}^\infty \theta_{n_ij}^2 = \sum_{j < C_ik_n} \theta_{n_ij}^2 + o(n^{-2r}) = \sum_{j < C_ik_n} \theta_{1n_ij}^2 + o(n^{-2r}).
\end{equation}
Hence, by (\ref{dub401}), we get $\|f_{2n_i}\| = o(n^{-r})$. We come to contradiction. \end{proof}

\begin{proof}[ Proof of Theorem \ref{tq6}. Necessary conditions] Let (\ref{con19}) do not hold. Then there are $\varepsilon >0$ and sequences $C_i \to \infty$, $n_i \to \infty$ as $i \to \infty$ such that
\begin{equation*}
\sum_{j > C_ik_{n_i}} \theta_{n_ij}^2 > \varepsilon n_i^{-2r} .
\end{equation*}
Then, by A4 and (\ref{u1}), we get
\begin{equation*}
n_i^2\sum_{j > C_ik_{n_i}} \kappa^2_{n_ij}\theta_{n_ij}^2 =o(1).
\end{equation*}
Therefore, by Theorem \ref{tq2}, subsequence $f_{1n_i} = \sum_{j > C_ik_{n_i}} \theta_{n_ij} \phi_j$ is inconsistent. \end{proof}

\begin{proof}[ Proof of Theorem \ref{tq12}] For proof of necessary conditions, it suffices to put $$f_{1n} = \sum_{j < C_1(\epsilon)k_n} \theta_{nj}\phi_j.$$  By Lemma \ref{ld3}, there is $\gamma_\epsilon$ such that $f_{1n} \in \gamma_\epsilon U$.
Proof of sufficiency  is simple and is omitted. \end{proof}

\begin{proof}[ Proof of Theorem \ref{tq8}] Necessary conditions are rather evident, and proof is omitted. Proof of sufficiency is also simple.
\begin{lemma}\label{ld5} Let for sequence $f_n$, $c n^{-r}<\| f_n \| < C n^{-r}$, (\ref{ma2}) hold. Then sequence $f_n$ is purely $n^{-r}$-consistent.
\end{lemma}
Suppose $f_n = \sum_{j = 1}^\infty \theta_{nj} \phi_j$  is not purely $n^{-r}$-consistent.
Then, by Theorem \ref{tq6}, there are $c_1$ and sequences $n_i$,  and $c_{n_i}$, $c_{n_i} \to \infty$ as $i \to \infty$, such that
\begin{equation*}
\sum_{j > c_{n_i} k_{n_i}} \theta_{n_lj}^2 >c_1 n_i^{-r}.
\end{equation*}
Therefore, if we put $f_{1n_i} = \sum_{j > c_{n_i} k_{n_i}} \theta_{n_ij} \phi_j $, then (\ref{ma2}) does not hold. \end{proof}
\subsection{Proof of Theorems of section \ref{sec5} \label{subsec9.5}}
    Denote
$$
T_{1n}(f) = T_{1n}(f,h_n) =\int_0^1\Bigl(\frac{1}{h_n}\int K\Bigl(\frac{t-s}{h_n}\Bigr)f(s)\, ds\Bigr)^2 dt.
$$
For sequence $\rho_n >0$, define  sets
$$
Q_{nh_n}(\rho_n) = \{f: T_{1n}(f) > \rho_n,\, f \in  \mathbb{L}_2^{per}(\mathbb{R}^1)\}.
$$

Define sequence of kernel-based tests $K_n ={\bf 1}_{\{T_n(Y_n) \ge x_\alpha\}}$,  $0 < \alpha <1$, with $x_\alpha$  defined the equation $\alpha = 1 - \Phi(x_\alpha)$.

Proof of Theorems is based on the following Theorem \ref{tk2} on asymptotic minimaxity of kernel-based tests $K_n$ (see Theorem 2.1.1  in  \cite{er03}).
\begin{thm}\label{tk2} Let $h_n^{-1/2}n^{-1} \to 0$, $h_n \to 0$ as $n \to \infty$.
Let
\begin{equation*}
0 < \liminf_{n \to \infty}  n \rho_n h_n^{1/2} \le \limsup_{n \to \infty} n\rho_nh_n^{1/2} < \infty.
\end{equation*}
Then sequence  of kernel-based tests $K_n $, is asymptotically minimax for the sets of alternatives $Q_{nh_n}(\rho_n)$.
There hold $\alpha(L_n) = \alpha(1 + o(1))$ and
\begin{equation}\label{33}
\beta(K_n,f_n)  = \Phi(x_\alpha - \gamma^{-1} \sigma^{-2}n h_n^{1/2} T_{1n}(f_n))(1 + o(1))
\end{equation}
uniformly onto sequences $f_n\in \mathbb{L}_2^{per}(R^1)$ such that $n h_n^{1/2} T_{1n}(f_n) < C $.
\end{thm}
We have
\begin{equation}\label{z33}
T_{1n}(f_n)   = \sum_{j=-\infty}^\infty |\hat K(jh_n)|^2 |\theta_{nj}|^2.
\end{equation}
Note that the unique difference  of setups of Theorems \ref{tk2} and \ref{tq2} is heteroscedastic noise. Thus Theorem \ref{tk2} can be obtained  by easy modification of the proof of Theorem \ref{tq2}.

If we put $|\hat K(jh_n)|^2 = \kappa_{nj}^2$, we get that the asymptotic (\ref{aq1}) in Theorem \ref{tq2} and the asymptotic (\ref{33}) coincide.
The function $\hat K(\omega)$, $\omega \in \mathbb{R}^1$, may have zeros. This cause the main differences in the statement of Theorems and in the reasoning. To clarify the differences we provide   proofs of sufficiency in  version of Theorem \ref{tk3} and {\sl iv.} in version of Theorem \ref{tq1}. Other proofs will be omitted.

\begin{proof}[ Proof of  version of Theorem \ref{tq3}. Sufficiency] Let (\ref{con2}) hold. We have
\begin{equation*}
\begin{split}&
T_{1n}(f_n)  = \sum_{j=-\infty}^\infty |\hat K(jh_n)|^2 |\theta_{nj}|^2 \ge \sum_{|j| h_n < b} |\hat K(jh_n)|^2 |\theta_{nj}|^2\\& \asymp
\sum_{|j| < c_2k_n} |\hat K(jh_n)|^2 |\theta_{nj}|^2
 \asymp  n^{-1}h_n^{-1/2} \asymp n^{-2r}
 \end{split}
\end{equation*}
for $c_2k_n < b h_n^{-1}$. By Theorem \ref{tk2}, this implies consistency.\end{proof}

\begin{proof}[ Proof of {\sl iv.} in version of Theorem \ref{tq1}]
Let $f = \sum_{j=-\infty}^\infty \tau_j \phi_j \notin \gamma U$ for all $\gamma >0$.  Then there is  sequence $m_l$, $m_l \to \infty$ as $l \to \infty$, such that
\begin{equation}\label{bb5}
m_l^{2s} \sum_{|j|\ge m_l}^\infty |\tau_j|^2 = C_l
\end{equation}
with $C_l \to \infty$ as $l \to \infty$.

It is clear that  we can define a sequence $m_l$ such that
\begin{equation}\label{gqq}
m_l^{2s} \sum_{m_l \le |j| \le 2m_{l}} |\tau_j|^2 > \delta C_l,
\end{equation}
where $\delta$, $0< \delta <1/2$,  does not depend on $l$.
Otherwise, we have
 \begin{equation*}
 2^{2s(i-1)} m_l^{2s} \sum_{j=2^{i-1}m_l}^{2^{i}m_l}  \tau_j^2  < \delta  C_l
\end{equation*}
 for  all $i = 1,2,\ldots$,  that implies that the left hand-side of
 (\ref{bb5}) does not exceed $2\delta C_l$.

Define a sequence $\etab_l = \{\eta_{lj}\}_{j=-\infty}^\infty$ such that
 $\eta_{lj} = \tau_j$,  $|j| \ge m_{l}$, and $\eta_{lj} = 0$ otherwise.

Denote
$$
\tilde f_l(x) =  \sum_{j=-\infty}^\infty \eta_{lj} \exp\{2\pi ijx\}.
$$
 For alternatives $\tilde f_l(x)$ we define $n_l$ such that $\|\tilde f_l(x)\| \asymp n_l^{-r}$.

 Then
\begin{equation*}
n_l \asymp C_l^{-1/(2r)} m_l^{s/r}.
\end{equation*}
We have $|\hat K(\omega)| \le \hat K(0) = 1$ for all $\omega \in R^1$ and $|\hat K(\omega)| > c > 0$ for $ |\omega| < b$. Hence, if we put $h_l= h_{n_l} =2^{-1}b^{-1}m_l^{-1}$, then, by (\ref{gqq}), there is $C > 0$  such that, for all $h> 0$, there holds
\begin{equation*}
T_{1n_l}(\tilde f_l,h_l) = \sum_{j=-\infty}^\infty |\hat K(jh_l)\,\eta_{lj}|^2 > C \sum_{j=-\infty}^\infty |\hat K(jh)\,\eta_{lj}|^2 = C T_{1n_l}(\tilde f_l,h).
\end{equation*}
Thus we can choose $h = h_l$ for further reasoning.

By (\ref{gqq}), we get
\begin{equation}\label{k101}
T_{1n_l}(\tilde f_l) =  \sum_{|j|>m_l}\, |\hat K(jh_l) \eta_{lj}|^2  \asymp \sum_{j=m_l}^{2m_l} |\eta_{lj}|^2 \asymp n_l^{-2r}.
\end{equation}
If we put in estimates (\ref{u7}),(\ref{u10}), $k_l = [h_{n_l}^{-1}]$ and $m_l = k_l$, then we get
\begin{equation}\label{k102}
h_{n_l}^{1/2} \asymp C_l^{(2r-1)/2}n_l^{2r-1}.
\end{equation}
By (\ref{k101}) and (\ref{k102}), we get
\begin{equation*}
n_l T_{1n_l}(\tilde f_l)h_{n_l}^{1/2} \asymp C_l^{-(1-2r)/2}.
\end{equation*}
By Theorem \ref{tk2}, this implies inconsistency of sequence of alternatives $\tilde f_l$. \end{proof}
    \subsection{Proof of Theorems of section \ref{sec6} \label{subsec9.6}}
  We have
\begin{equation*}
n^{-1}m_n^{-1}T_n(F) = \sum_{l=0}^{m_n-1}\left(\int_{l/m_n}^{(l+1)/m_n} f(x) dx \right)^2.
\end{equation*}
Using representation  $f(x)$ in terms of Fourier coefficients
\begin{equation*}
f(x) = \sum_{j=-\infty}^\infty \theta_j \exp\{2\pi ijx\},
\end{equation*}
we get
\begin{equation*}
\int_{l/m_n}^{(l+1)/m_n} f(x) dx = \sum_{j=-\infty}^\infty \frac{\theta_j}{2\pi ij}\exp\{2\pi ijl/m_n\} ( \exp\{2\pi ij/m_n\} - 1)
\end{equation*}
for $ 1 \le l < m_n$.

In what follows, we shall use  the following agreement $0/0= 0$.
\begin{lemma}\label{ch1} There holds
\begin{equation}\label{ach1}
n^{-1}m_n^{-1}T_n(F) = m_n \sum_{k=-\infty}^\infty \sum_{j \ne km_n} \frac{\theta_j \bar\theta_{j-km_n}}{4\pi^2j(j-km_n)}(2 - 2 \cos(2\pi j/m_n)).
\end{equation}
\end{lemma}
\begin{proof}[Proof of Lemma \ref{ch1}] We have
\begin{equation}\label{uh105}
\begin{split}&
n^{-1}m_n^{-1}T_n(F) =\sum_{l=0}^{m_n-1}\Bigl(\sum_{j \ne 0}\frac{\theta_j}{2\pi ij}\exp\{2\pi ijl/m_n\} ( \exp\{2\pi ij/m_n\} - 1)\Bigr)\\& \times
\Bigl(\sum_{j \ne 0} \frac{-\bar\theta_j}{2\pi ij}\exp\{-2\pi ijl/m_n\} (\exp\{-2\pi ij/m_n\} - 1)\Bigr) = J_1 + J_2
\end{split}
\end{equation}
with
\begin{equation}\label{uh106}
\begin{split}&
J_1=\sum_{l=0}^{m_n-1}\sum_{k=-\infty}^\infty\,\,\, \sum_{j_1 = j -km_n} \frac{\theta_j \bar\theta_{j_1}}{4\pi^2jj_1}\exp\{2\pi ilk\}\\&\times(\exp\{2\pi ij/m_n\} - 1)(\exp\{-2\pi ij_1/m_n\} - 1)\\&= m_n \sum_{k=-\infty}^\infty \sum_{j=-\infty}^\infty \frac{\theta_j \bar\theta_{j-k m_n}}{4\pi^2j(j-k m_n)}(2 - 2 \cos(2\pi j/m_n))
\end{split}
\end{equation}
and
\begin{equation}\label{ux}
\begin{split}&
J_2=\sum_{l=0}^{m_n-1}\sum_{j \ne 0} \sum_{j_1 \ne j-k m_n} \frac{\theta_j \bar\theta_{j_1}}{4\pi^2 j j_1} \exp\{2\pi i(j - j_1)l/m_n\}\\&\times(\exp\{2\pi ij/m_n\} - 1)(\exp\{-2\pi ij_1/m_n\} - 1) =0,
\end{split}
\end{equation}
where $j_1 \ne j-k m_n$ signifies that summation is performed over all $j_1$ such that $j_1 \ne j-k m_n$ for all integer $k$.

In the last equality of  (\ref{ux}), we make use of the identity
\begin{equation*}
 \sum_{l=0}^{m_n-1}\exp\{2\pi i(j - j_1)l/m_n\} = \frac{\exp\{2\pi i(j - j_1) m_n/m_n\} -1}{\exp\{2\pi i(j - j_1)/m_n\} -1} =0,
\end{equation*}
if $j-j_1 \ne k m_n$ for all integer $k$.

By (\ref{uh105}) - (\ref{ux}) together, we get (\ref{ach1}). \end{proof}

For any c.d.f $F$ and any $k$ denote $\tilde F_k$ the function having the derivative
$$
 1 + \tilde f_k(x)  = 1 + \sum_{|j| > k} \theta_j \exp\{2\pi ijx\}.
$$
and such that $\tilde F_k(1) = 1$.

 Denote $i_n = [d m_n]$ where $d > 1+c$.
\begin{lemma}\label{lch2} There holds
\begin{equation}\label{hu}
n^{-1}m_{n}^{-2}T_{n}(\tilde F_{i_n})  \le C m_n^{-1} i_n^{-1}\sum_{|j| > i_n} |\theta_j|^2.
\end{equation}
\end{lemma}
\begin{proof} Denote $\eta_j = \theta_j$ if $|j|> i_n$ and $\eta_j =0$ if $|j| < i_n$.

We have
\begin{equation*}
\begin{split}&
n^{-1}m_{n}^{-2}T_{n}(\tilde F_{i_n}) =\sum_{k=-\infty}^\infty \,\,\sum_{j \ne k m_n}
\frac{\eta_j \bar\eta_{j-k m_n}}{4\pi^2j(j-k m_n)}(2 -2 \cos(2\pi j/m_n))\\&
\le C  \sum_{|j|> i_n} \Bigl|\frac{\eta_j}{j}\Bigr| \sum_{k=-\infty}^\infty  \Bigl|\frac{\eta_{j+k m_n}}{j+k m_n}\Bigr|
\\&
= C\sum_{j=1}^{m_n} \sum_{k=-\infty}^\infty \Bigl|\frac{\eta_{j+k m_n}}{j+k m_n}\Bigr|  \sum_{k_1=-\infty}^\infty
\Bigl|\frac{\eta_{j+(k+k_1)m_n}}{j+(k+k_1)m_n}\Bigr| \\&
 = C\sum_{j=1}^{m_n} \Bigl(\sum_{k=-\infty}^\infty \Bigl|\frac{\eta_{j+k m_n}}{j+k m_n}\Bigr|\Bigr)^2\\&
\le  C\sum_{j=1}^{m_n} \Bigl(\sum_{|k|> d-1} |\eta_{j+k m_n}|^2\Bigr)
  \Bigl(\sum_{|k|> d-1} (j+k m_n)^{-2}\Bigr)
\\& \le C \sum_{j=-\infty}^\infty  |\eta_j|^2 \sum_{|k| > d}(k m_n)^{-2}
 \le C m_n^{-1} i_n^{-1}\sum_{|j| > i_n} |\theta_j|^2.
\end{split}
\end{equation*}
\end{proof}

\begin{proof}[ Proof of version of Theorem \ref{tq3}. Sufficiency] Let (\ref{con2}) hold.
Denote
$$
\tilde f_{n}  = \tilde f_{n,c_2k_n}  =\sum_{|j| > c_2 k_n} \theta_{nj}\phi_j\quad\mbox{and}\quad \bar f_{n} = \bar f_{n,c_2k_n} = f_n - \tilde f_{n}
$$
Denote $\tilde F_{n}$, $\bar F_{n}$  the functions having derivatives $1+\tilde f_{n,c_2k_n}$ and $1+\bar f_{n,c_2k_n}$ respectively and such that $\tilde F_{n}(1) = 1$,   and $\bar F_{n}(1) =1$.

Let $T_n$ be  chi-squared test statistics with a number of cells $m_n= [c_3 k_n]$  where $c_2 < c_3 $.
Denote $\mathbb{L}_{2,n}$ linear space generated functions ${\bf 1}_{\{x \in ((j-1)/m_n,j/m_n)\}}$, $1  \le j \le m_n$.

Denote $\bar h_n$  orthogonal projection of $\bar f_n$  onto $\mathbb{L}_{2,n}$. Denote $\tilde h_n$ orthogonal projection of  $\tilde f_n$ onto the line $\{h\,:\, h = \lambda \bar h_n,\, \lambda \in \mathbb{R}^1\}$.

Note that $n^{-1/2} T_n^{1/2}(F_n)$ equals the $\mathbb{L}_{2,n}$-norm of $f_{n}$. Therefore we have
\begin{equation}\label{ee9}
n^{-1/2}\, T_n^{1/2}(F_n) \ge  \|\bar h_n + \tilde h_n\|.
\end{equation}
Hence, by Theorem \ref{chi2}, it suffices to show that, for some choice of $c_3$, there holds $\|\bar h_n + \tilde h_n\| \asymp n^{-r}$ if $m_n > c_3\, k_n$.

Denote $\bar g_n = \bar f_n - \bar h_n$ and $\tilde g_n = \tilde f_n - \tilde h_n$.

Denote
$$
\bar p_{jn} = \frac{1}{m_n} \int_{(j-1)/m_n}^{j/m_n} \bar f_n(x) dx, \quad 1 \le j \le m_n.
$$
By Lemmas 3 and 4 in section 7 of \cite{ul}, we have
\begin{equation}\label{h5}
 \|\bar g_n\|^2 = m_n\sum_{j=1}^{m_n} \int_{(j-1)/m_n}^{j/m_n}(\bar f_n(x) - \bar p_{jn})^2 \, dx \le 2 \omega^2\Bigl(\frac{1}{m_n}, \bar f_n\Bigr).
\end{equation}
Here
$$\omega^2(h,f) = \int_0^1 (f(t+h) - f(t))^2\,dt, \quad h>0,$$
for any $f \in \mathbb{L}_2^{per}$.
If $f = \sum_{j=-\infty}^\infty \theta_j \phi_j,
$ then
\begin{equation}\label{h6}
\omega^2(s,f) =  2\sum_{j=1}^\infty |\theta_j|^2\,(2  - 2\cos (2\pi js)).
\end{equation}
 Since $1 - cos( x) \le x^2$, then, by (\ref{h5}) and (\ref{h6}), we have
\begin{equation}\label{h8}
\|\bar g_n\|  \le\, 4 \pi (c_2\,k_n/m_n )^{1/2}\, \|\bar f_n\|= \delta \|\bar f_n\|(1+ o(1)),
\end{equation}
where $\delta = 4\pi\,(c_2/c_3)^{1/2}$.

For any functions $g_1, g_2 \in \mathbb{L}_2(0,1)$ denote $(g_1,g_2)$ inner product of $g_1$ and $g_2$.

We have
\begin{equation}\label{ee10}
0 = (\bar f_n, \tilde f_n) = (\bar h_n, \tilde h_n) + (\bar g_n, \tilde f_n).
\end{equation}
By (\ref{h8}), we get
\begin{equation*}
|(\bar g_n, \tilde f_n)| \le \|\bar g_n\|\, \|\tilde f_n\| \le \delta C^2 n^{-2r}.
\end{equation*}
Therefore we get
\begin{equation}\label{ee11}
|(\bar h_n, \tilde h_n) | \le \delta C^2 n^{-2r}.
\end{equation}
By (\ref{h8}) - (\ref{ee11}), we get that, for sufficiently small $\delta$ there holds $\|\bar h_n + \tilde h_n\| \asymp n^{-r}$. \end{proof}

\begin{proof}[ Proof of version of Theorem \ref{tq4}. Sufficiency] Let $k_n= [c_1 n^{2 - 4r}]$.
For $c_2 > 2c_1$, we have
\begin{equation}\label{h11}
 T_n^{1/2}(F_{n}) \le T_n^{1/2}(\bar F_{n}) + T_n^{1/2}(\tilde F_{n}).
\end{equation}
By Lemma \ref{lch2}, we have
\begin{equation}\label{h12}
n^{-1} T_n(\tilde F_{n}) \le c_2^{-1}m_n k_n^{-1} \|\tilde f_n \|^2  \le c_2^{-1} c_1 C n^{-2r}.
\end{equation}
We have
\begin{equation}\label{h13}\|\bar f_n\| \ge n^{-1/2} T_n^{1/2}(\bar F_n).
\end{equation}
Since one can take arbitrary value $c_2$, $c_2 > 2c_1$, then, by Theorem \ref{chi2}, (\ref{con3}) and  (\ref{h11}) - (\ref{h13}) together, we get
inconsistency of sequence $f_n$. \end{proof}

\begin{proof}[ Proof of  iv. in version of Theorem \ref{tq1}]
 Suppose  opposite. Then there is sequence $i_l$, $i_l \to \infty$ as $l \to \infty$, such that
\begin{equation*}
i_l^{2s} \|\tilde f_{i_l}\|^2  = C_l,
\end{equation*}
with $C_l \to \infty$ as $l \to \infty$. Here $f = \sum_{j=-\infty}^\infty\tau_j \phi_j$ and $\tilde f_{i_l} = \sum_{|j| >i_l}\tau_j \phi_j$.

Let $n_l$ be such that $n_l^{-r} \asymp \|\tilde f_{i_l}\|$.

Then, estimating similarly to  (\ref{u7}) and (\ref{u10}), we get $i_{l}^{-1/2} \asymp C_l^{(2r-1)/2}n_l^{2r-1}$.

If $m_l = o(i_l)$, then, by Lemma \ref{lch2}, we get
\begin{equation}\label{xy1}
m_{l}^{-1/2}T_{n_l}(\tilde F_{i_l}) \le m_{l}^{1/2} i_l^{-1} n_l \sum_{|j| > i_l} |\tau_j|^2 \asymp m_{l}^{1/2} i_l^{-1} n_l^{1-2r} = o( C_l^{(2r-1)/2}).
\end{equation}

Let $m_l \asymp i_l$ or $i_l = o(m_l)$. Then we have
\begin{equation}\label{udav1}
n_l^{-2r} \asymp \|\tilde f_{i_l}\|^2  \ge n_l^{-1} T_{n_l}(\tilde F_{i_l}).
\end{equation}
Therefore
\begin{equation}\label{xy2}
m_l^{-1/2} T_{n_l}(\tilde F_{i_l}) \le Cm_l^{-1/2}n_l^{1-2r} = C m_l^{-1/2}i_l^{1/2} C_l^{(2r-1)/2}  = o(1).
\end{equation}
By  Theorem \ref{chi2}, (\ref{xy1}) -(\ref{xy2}) imply {\sl iv.} \end{proof}

\begin{proof}[ Proof of version of Theorem \ref{tq11}] Let
$
f_{1n} = \sum_{|j|< ck_n} \theta_{nj} \phi_j.
$
Then, by Lemma \ref{ld3}, there is $\gamma$ such that $f_{1n} \in \gamma U$.

Denote $F_{1n}$ function having derivative $1 + f_{1n}$ and such that $F_{1n}(1) = 1$.

We have
\begin{equation}\label{uxa1}
|T_n^{1/2}(F_n) - T_n^{1/2}(F_{1n})| \le T_n^{1/2}(F_n - F_{1n} + F_0).
\end{equation}
If $m_n = [c_0 k_n]$  and $c > 2c_0$, then, by Lemma \ref{lch2}, we have
\begin{equation}\label{uxa2}
n^{-1} T_n(F_n - F_{1n} + F_0)  \le c_0 c^{-1}\, \| f_n - f_{1n}\|^2.
\end{equation}
Since the choice of $c$ is arbitrary, by Theorem \ref{chi2}, (\ref{uxa1}) and (\ref{uxa2}) imply (\ref{uuu}) and (\ref{uu1}). \end{proof}

  Proof of  {\sl iii.} in version of Theorem \ref{tq1} and versions of Theorems \ref{tq7}, \ref{tq6}, \ref{tq12}, \ref{tq8} follows from Theorem \ref{chi2} and versions of Theorems \ref{tq3} and \ref{tq4}  using the same reasoning as in  subsection \ref{subsec9.4}.

 \begin{proof}[ Proof of version of Theorem \ref{tq5}]
 Denote  $F_{1n}$ c.d.f. having the density $1 + f_{1n}$.

 Let $f_{1n} = \sum_{j=-\infty}^\infty \eta_{nj} \phi_j$.
 For any $a >0$ denote
 $$
 \bar f_{a1n} =\sum_{|j| < ak_n} \eta_{nj} \phi_j\quad \mbox{and}\quad  \tilde f_{a1n} = f_{1n} - \bar f_{a1n}.
 $$
  Define functions $\bar F_{a,1n}(x)$, $\tilde F_{a,1n}(x)$ with $x \in [0,1]$ such that
 $1 + \bar f_{a,1n}(x) = d \bar F_{a,1n}(x)/dx$, $\tilde f_{a,1n}(x) = d \tilde F_{a,1n}(x)/dx$ and $1 + \bar F_{a,1n}(1) =1$,  $\bar F_{a,1n}(1) =1$.

  We have
 \begin{equation}\label{ch121}
 \begin{split}&
 T_n^{1/2}(F_n+ F_{1n}-F_0) \le T_n^{1/2}(F_n) + T_n^{1/2}(F_{1n})\\& \le T_n^{1/2}(F_n) + T_n^{1/2}(\bar F_{a,1n}) + T_n^{1/2}(\tilde F_{a1n})
 \end{split}
 \end{equation}
and
 \begin{equation}\label{ch200}
T_n^{1/2}(F_n+ F_{1n}- F_0) \ge T_n^{1/2}(F_n) - T_n^{1/2}(\bar F_{a,1n}) - T_n^{1/2}(\tilde F_{a1n}).
 \end{equation}
 Therefore, by Theorem \ref{chi2}, it suffices to estimate $T_n^{1/2}(\bar F_{a,1n})$ and $T_n^{1/2}(\tilde F_{a1n})$.

We  have
 \begin{equation}\label{ch3}
m_n^{-1/2}T_n(\bar F_{a,1n}) \le nm_n^{-1/2} \|\bar f_{a,1n}\|^2 = o(n^{1-2r} m_n^{-1/2}) = o(1).
 \end{equation}
 By Lemma \ref{lch2}, we have
 \begin{equation}\label{ch4}
m_n^{-1/2} T_n(\tilde F_{a,1n}) \le Ca^{-1}nm_n^{-1/2} \|\tilde f_{a,1n}\|^2 = O(a^{-1}n^{1-2r} m_n^{-1/2}) = O(a^{-1}).
 \end{equation}
 By Theorem \ref{chi2}  and (\ref{ch121}) - (\ref{ch4}) together, we get  version of Theorem \ref{tq5}. \end{proof}
    \subsection{Proof of Theorems of section \ref{sec7} \label{subsec9.7}}
    Lemma \ref{lc1} given below  allows to carry over corresponding reasoning for Brownian bridge $b(t)$, $t \in (0,1)$, instead of empirical distribution functions.
 \begin{lemma}\label{lc1} For any $x > 0$, we have
\begin{equation}\label{lm1}
{P}_{F_n}(nT^2(\hat F_n - F_0) < x) -  {P}\,(T^2(b(t) + \sqrt{n}(F_n(t) - F_0(t))) < x)  = o(1)
 \end{equation}
 uniformly onto $F_n$ such that $T(F_n - F_0)  < cn^{-1/2}$.
 \end{lemma}
 If $\sqrt{n}(F_n - F_0) \to G$  in Kolmogorov - Smirnov distance, (\ref{lm1}) has been proved Chibisov \cite{chib} without any statements of uniform convergence.

 Lemma \ref{lc1} follows from Lemmas \ref{lc2} and \ref{lc4} given below after implementation of Hungary construction (see Th. 3, Ch. 12,  section 1, \cite{wel}).

 \begin{lemma}\label{lc2} For any $x > 0$, we have
 \begin{equation}\label{lm2}
 \begin{split}&
 \mathbf{ P}\,(T^2(b(F_n(t)) + \sqrt{n}(F_n(t) - F_0(t))) < x) \\& -  {\mathbf P}\,(T^2(b(t) + \sqrt{n}(F_n(t) - F_0(t))) < x)  = o(1)
 \end{split}
 \end{equation}
 uniformly onto $F_n$ such that $T(F_n - F_0)  < cn^{-1/2}$.
 \end{lemma}
 Lemma \ref{lc2} follows from Lemmas \ref{lc3} and \ref{lc4} given below.
 \begin{lemma}\label{lc3} There holds
 \begin{equation}\label{lm100}
 {\mathbf E}\,[|T^2(b(F_n(t)))- T^2(b(t))|]  < c T^{1/4}(F_n - F_0).
 \end{equation}
\end{lemma}
\begin{proof} We have
\begin{equation}\label{lm3}
\begin{split}&
 \mathbf{E}^2\, [\,|T^2(b(F_n(t))- T^2(b(t))|] \le \mathbf{E}^2\,[|(T(b(F_n(t))- T(b(t)))\,(T(b(F_n(t)) + T(b(t)))|]\\&
\le  \mathbf{E}\,[((T(b(F_n(t)))- T(b(t)))^2]\,\mathbf{E}\, [(T(b(F_n(t))) + T(b(t)))^2]\\& \le
C  \mathbf{E}\, [((T(b(F_n(t))- T(b(t)))^2]  \le C \mathbf{E}\, [T^2(b(F_n(t)) -b(t)))] \\& =
\int_0^1 (F_n(t) - F^2_n(t) - 2\min (F_n(t), F_0(t)) + 2F_n(t)F_0(t) + F_0(t) - F_0^2(t) \, dt\\& = \int_0^1 F_n(t) + F_0(t) -2 \min(F_n(t), F_0(t)) - (F_n(t) - F_0(t))^2 \,dt\\&
= \int_0^1 |F_n(t) - F_0(t)| - (F_n(t) - F_0(t))^2 \,dt\\& \le
\int_0^1\, |F_n(t) - F_0(t)| \, dt\, \le\, T^{1/2}(F_n - F_0).
\end{split}
\end{equation} \end{proof}
\begin{lemma}\label{lc4}
Densities of c.d.f.'s $\mathbf{P}\,(T^2(b(t) + n^{1/2}(F_n(t) - F_0(t))) \le x)$ are uniformly bounded onto the set of all c.d.f. $F_n$ such that $nT^2(F_n -F_0) <C$.
\end{lemma}
\begin{proof} Brownian bridge $b(t)$  admits representation
\begin{equation*}
b(t) = \sum_{j=1}^\infty \frac{\xi_j}{\pi j} \psi_j(t)
\end{equation*}
where $\psi_j(t) = \sqrt{2}\,\sin(\pi j t)$ and $\xi_j$, $1 \le j < \infty$, are i.i.d. Gaussian random variables, $\mathbf{E}\, \xi_j = 0$ and $ \mathbf{E}\, \xi_j^2 = 1$.

Therefore, if $f_n(t) = \sum_{j=1}^\infty \theta_{nj} \phi_j$, then
\begin{equation}\label{lm4}
T^2(b(t) + n^{1/2}(F_n(t) - F_0(t))) = \sum_{j=1}^\infty \frac{(\xi_j + n^{1/2} \theta_{nj})^2}{\pi^2 j^2}.
\end{equation}
The right hand-side of (\ref{lm4}) is a sum of independent random variables. Thus it suffices to show that
\begin{equation*}
(\xi_1 + n^{1/2} \theta_{n1})^2 + \frac{1}{4}(\xi_2 + n^{1/2} \theta_{n2})^2
\end{equation*}
has bounded densities uniformly onto $n^{1/2}|\theta_{n1}| \le C$ and
 $n^{1/2}|\theta_{n2}| \le C$ for any $C$.

 Densities $(\xi_1 + a)^2$ and $(\xi_2 + b)^2$ with $|a| \le C$ and
 $|b| \le C$ have wellknown analytical form and proof of boundedness of density of $(\xi_1 + a)^2 + \frac{1}{4}(\xi_2 + b)^2$ is obtained by routine technique. We omit these standard estimates. \end{proof}

 \begin{proof}[ Proof of {\sl ii.} in Theorem \ref{tcm}] Hungary construction allows to reduce reasoning  to proof of corresponding statement for Brownian bridge $b(t)$, $t \in [0,1)]$. This reasoning is provided in supplement. Thus it suffices to prove the following Lemma.
 \begin{lemma}\label{lum1} There holds
 \begin{equation}\label{lhm1}
\ \limsup_{n \to \infty} \sup_{F \in \Psi_n(a)} \mathbf{P}\,(T^2(b(t) + \sqrt{n}\,(F(t) - F_0(t))) \le x_\alpha) < 1 - \alpha,
\end{equation}
where $\Psi_n(a)= \{\, F\,:\, nT^2(F- F_0) > a,\, nT^2(F- F_0) < C ,\, F \mbox { is c.d.f.} \}$, $ C > a$. Here $x_\alpha$ is assigned by equation $\mathbf{P} (T^2(b(t)) > x_\alpha) = \alpha$.
\end{lemma}
\begin{proof} Suppose opposite that (\ref{lhm1}) does not valid. Then there is subsequence c.d.f.'s $F_{n_i} \in \Psi_{n_i}(a)$, $n_i \to \infty$ as $i \to \infty$ such that
\begin{equation} \label{lhm5}
\lim_{i \to \infty}\mathbf{P}\,(T^2(b(t) + \sqrt{n_i}\,(F_{n_i}(t) - F_0(t))) \le x_\alpha) \ge  1 - \alpha,
\end{equation}
where $dF_{n_i}(x)/dx = 1 + \sum_{j=1}^\infty \theta_{n_ij} \phi_j(x)$, $x \in (0,1)$, and $F_{n_i}(0) = 0$.

There are $\etab = \{ \eta_j \}_{j=1}^\infty$ and subsequence $n_{i_k}$ of sequence $n_i$ such that
$n^{1/2} \theta_{n_{i_k}j} \, j^{-1} \to \eta_j$ as $k \to \infty$ for each $j$, $1 \le j < \infty$.

Therefore there is $C_k$, $C_k \to \infty$ as $k \to \infty$, such that
\begin{equation} \label{lhm6}
\lim_{k \to \infty}\frac{n_{i_k}\, \sum_{j < C_k} \theta_{n_{i_k}j}^2 j^{-2}}{ \sum_{j < C_k} \eta_j^2 } = 1
\end{equation}
and
\begin{equation} \label{lhm206}
\lim_{k \to \infty}\, \sum_{j < C_k} (n_{i_k}^{1/2}\theta_{n_{i_k}j} j^{-1} -\eta_j)^2 = 0
\end{equation}

We consider two cases.
\vskip 0.15cm
{\it i.} There holds
\begin{equation*}
\lim_{k \to \infty} n_{i_k} \, \sum_{j > C_k} \theta_{n_{i_k}j}^2 j^{-2} = 0.
\end{equation*}
\vskip 0.15cm
{\it ii.} There holds
\begin{equation*}
 n_{i_k} \, \sum_{j > C_k} \theta_{n_{i_k}j}^2 j^{-2} > c \quad \mbox{for} \quad k > k_0.
 \end{equation*}
 If {\it i.} holds, we have
 \begin{equation}\label{uxx1}
 n_{i_k}\,\mathbf{E}\,\Bigl(\sum_{j > C_k} \xi_j\,\theta_{n_{i_k}j}\,j^{-2} \Bigr)^2 = n_{i_k} \sum_{j > C_k} \theta_{n_{i_k}j}^2\,j^{-4} \le C_k^{-2} n_{i_k} \sum_{j > C_k} \theta_{n_{i_k}j}^2\,j^{-2} = o(1).
 \end{equation}
 By (\ref{lhm206}), we get
 \begin{equation}\label{uxx2}
 \mathbf{E} \,\left(\sum_{j < C_k} \xi_j\,(n_{i_k}^{1/2}\theta_{n_{i_k}j} j^{-1} -\eta_j)\right)^2 = \sum_{j < C_k} (n_{i_k}^{1/2}\theta_{n_{i_k}j} j^{-1} -\eta_j)^2 = o(1).
 \end{equation}
By (\ref{uxx1}) and (\ref{uxx2}), we get
 \begin{equation*}
 \begin{split}&
\mathbf{P}\Bigl(\pi^{-2}\,\sum_{j=1}^\infty (\xi_j + n_{i_k}^{1/2} \theta_{n_{i_k}j})^2 \, j^{-2} < x_\alpha\,\Bigr)\\& = \mathbf{P}\Bigl(\pi^{-2}\,\sum_{j < C_k} (\xi_j + n_{i_k}^{1/2} \theta_{n_{i_k}j})^2 \, j^{-2} + \pi^{-2}\,\sum_{j > C_k} \xi_j^2\, j^{-2} < x_\alpha\,(1 + o_P(1))\,\Bigr)\\& = \mathbf{P}\Bigl(\pi^{-2}\,\sum_{j < C_k}(\xi_j\,j^{-1} + \eta_j)^2  + \pi^{-2}\,\sum_{j > C_k} \xi_j^2\, j^{-2} < x_\alpha\,(1 + o_P(1))\,\Bigr)\\& <
\,\mathbf{P}\Bigl(\pi^{-2}\,\sum_{j=1}^\infty \xi_j^2\, j^{-2} < x_\alpha\,\Bigr)\,(1 + o(1)).
\end{split}
 \end{equation*}
where the last inequality follows from Lemma \ref{lum2} given below.
\begin{lemma}\label{lum2} Let $\eta = \{\eta_j\}_1^\infty$ be such that $\pi^{-2}\,\sum_{j=1}^\infty \eta_j^2  > c$. Then there holds
 \begin{equation}\label{lhm2}
 \mathbf{P}\,\Bigl(\pi^{-2}\,\sum_{j=1}^\infty \xi_j^2\, j^{-2} < x_\alpha\Bigr) > \mathbf{P}\,\Bigl(\pi^{-2}\,\sum_{j=1}^\infty (\xi_j/j + \eta_j)^2  < x_\alpha\Bigr).
 \end{equation}
 \end{lemma}
 \begin{proof} For simplicity of notation the reasoning will be provided for $\eta_1 \ne 0$. Implementing Anderson Theorem \cite{an}, we get
\begin{equation}\label{lhm4}
\begin{split}&
\mathbf{P}\,\Bigl(\pi^{-2}\,\sum_{j=1}^\infty (\xi_j/j + \eta_j)^2 < x_\alpha \Bigr)\\& =
(2\pi)^{-1/2} \int_{-\pi\sqrt{x_\alpha}- \eta_1}^{\pi\sqrt{x_\alpha}- \eta_1} \exp \Bigl\{-\frac{x^2}{2}\Bigr\} \mathbf{P}\Bigl(\pi^{-2}\,\sum_{j=2}^\infty (\xi_j/j + \eta_j)^2 < x_\alpha - \pi^{-2}\,(x + \eta_1)^2\Bigr)\, d\,x\\& \le
(2\pi)^{-1/2} \int_{-\pi\sqrt{x_\alpha}- \eta_1}^{\pi\sqrt{x_\alpha}- \eta_1} \exp\Bigl\{-\frac{x^2}{2}\Bigr\} \mathbf{P}\Bigl(\pi^{-2}\,\sum_{j=2}^\infty \xi_j^2 \, j^{-2} < x_\alpha - \pi^{-2}\,(x + \eta_1)^2\Bigr)\, d\,x\\& =
\mathbf{P}\,\Bigl(\, \pi^{-2}\,(\xi_1 + \eta_1)^2 + \pi^{-2}\,\sum_{j=2}^\infty \xi_j^2 \,j^{-2} < x_\alpha\, \Bigr) <  \mathbf{P}\,\Bigl(\pi^{-2}\,\sum_{j=1}^\infty \xi_j^2\, j^{-2} < x_\alpha\Bigr).
\end{split}
\end{equation}
For the proof of last inequality in (\ref{lhm4}) it suffices to note that $\mathbf{P}( \xi_1^2 < x) > \mathbf{P}( (\xi_1 + \eta_1)^2 < x)$ for $x \in (0, x_\alpha)$, and, for any $\delta$, $0 < \delta < x_\alpha$, there is $\delta_1 > 0$ such that the function $\mathbf{P}( \xi_1^2 < x) - \mathbf{P}( (\xi_1 + \eta_1)^2 < x)- \delta_1$ is positive onto interval $(\delta, x_\alpha)$.\end{proof}

Suppose {\it ii.} holds. Then we have
\begin{equation}\label{lum5}
\begin{split}&
T^2(b(t) + \sqrt{n}(F_n(t) - F_0(t))) = \sum_{j=1}^\infty \frac{(\xi_j + n^{1/2} \theta_{nj})^2}{\pi^2 j^2}\\& = \sum_{j< C_n}\frac{(\xi_j + n^{1/2} \theta_{nj})^2}{\pi^2 j^2} +
\sum_{j \ge C_n}\frac{(\xi_j + n^{1/2} \theta_{nj})^2}{\pi^2 j^2} = J_{1n} + J_{2n}.
\end{split}
\end{equation}
We have
\begin{equation}\label{lum6}
\begin{split}&
J_{2n}  = \sum_{j \ge C_n }\,\frac{\xi_j^2}{\pi^2 j^2} + 2\sqrt{n}\sum_{j \ge C_n } \, \frac{\xi_j\theta_{nj}}{\pi^2 j^2}\\& + n\sum_{j \ge C_n}\,\frac{\theta_{nj}^2}{\pi^2 j^2} = J_{21n} + 2 J_{22n}  +J_{23n}
\end{split}
\end{equation}
We have
\begin{equation}\label{lum7}
J_{21n} = o_P(1) \quad \mbox{and}\quad J_{22n} \le J_{21n}^{1/2}\, J_{23n}^{1/2} = o_P(1).
\end{equation}
By (\ref{lum5}) - (\ref{lum7}), implementing Anderson Theorem \cite{an}, we get that, for any $\delta> 0$, there holds
\begin{equation}\label{lum9}
\begin{split}&
\mathbf{P}\Bigl(\sum_{j=1}^\infty \frac{(\xi_j + n^{1/2} \theta_{nj})^2}{\pi^2 j^2} < x\Bigr)
\le \mathbf{P}\Bigl(\sum_{j< C_n}  \frac{(\xi_j + n^{1/2} \theta_{nj})^2}{\pi^2 j^2}  \le x - c - o_P(1)\Bigr)\\&
\le \mathbf{P}\Bigl(\sum_{j< C_n} \frac{\xi_j^2}{\pi^2 j^2}  \le x - c + \delta\Bigr)(1+ o(1))
\le \mathbf{P}\Bigl(\sum_{j=1}^\infty \frac{\xi_j^2}{\pi^2 j^2}  \le x - c + 2\delta\Bigr)(1+ o(1)).
\end{split}
\end{equation}
\end{proof} \end{proof}

\begin{proof}[ Proof of version of Theorem  \ref{tq3}]
Let (\ref{con2}) hold. Then we have
\begin{equation*}
n\sum_{j=1}^\infty\frac{\theta_{nj}^2}{\pi^2 j^2} \ge n\sum_{j< c_2 k_n}\frac{\theta_{nj}^2}{\pi^2 j^2} \ge c_2^{-2} n k_n^{-2}\sum_{j< c_2 k_n}\theta_{nj}^2 \asymp 1.
\end{equation*}
By (\ref{cru}), this implies sufficiency. \end{proof}

\begin{proof}[ Proof of version of Theorem  \ref{tq4}] Let (\ref{con3}) hold. Then we have
\begin{equation}\label{om202}
\begin{split}&
n\sum_{j=1}^\infty\frac{\theta_{nj}^2}{\pi^2 j^2} = n\sum_{j<c_2k_n}\frac{\theta_{nj}^2}{\pi^2 j^2} + n\sum_{j>c_2k_n}\frac{\theta_{nj}^2}{\pi^2 j^2}\\& \le  o(1) + (c_2 k_n)^{-2}
n \sum_{j > c_2 k_n} \theta_{nj}^2 \asymp o(1) +  (c_2k_n)^{-2} n^{1-2r} = O(c_2^{-2}).
\end{split}
\end{equation}
Since $c_2$  is arbitrary,  then, by (\ref{cru}), (\ref{om202}) implies sufficiency. \end{proof}

\begin{proof}[ Proof of iii. in Theorem \ref{tom1}] The reasoning are akin to  proof of {\sl iii.} in Theorem  \ref{tq1}. The statement follows  from (\ref{con2}) and Lemma \ref{lom25} provided below.

\begin{lemma} \label{lom25} Let $f_n \in c_1 U$  and $cn^{-r}\le \|f_n\| \le Cn^{-r}
$. Then, for $k_n = C_1 n^{(1-2r)/2}(1  +o(1))$  with $C_1^{2s} > 2c_1/c$, there holds
\begin{equation*}
\sum_{j=1}^{k_n} \theta_{nj}^2 > \frac{c}{2} n^{-2r}.
\end{equation*}
\end{lemma}
 Proof of Lemma \ref{lom25} is akin to proof of Lemma \ref{ld1} and is omitted. \end{proof}

\begin{proof}[ Proof of iv. in Theorem \ref{tom1}] Reasoning is akin to  proof of {\sl iv.} in Theorem  \ref{tq1}. Suppose  opposite. Then there are $f = \sum_{j=1}^\infty \tau_{j}\,\phi_j  \notin \mathbb{B}^s_{2\infty}$,  and a sequence $m_l, m_l \to \infty$ as $l \to \infty$, such that (\ref{u5}) holds. Define sequences $\etab_l$, $n_l$ and $\tilde f_l$ by the same way as in the proof of Theorem \ref{tq1}.
Then
\begin{equation*}
n_l \asymp C_l^{-1/(2r)} m_l^{s/r} = C_l^{-1/(2r)} m_l^{\frac{2}{1 - 2r}}.
\end{equation*}
Therefore we get
\begin{equation*}
m_l \asymp C_l^{(1-2r)/(4r)}n_l^{\frac{1-2r}{2}}.
\end{equation*}
Hence we get
\begin{equation}\label{omu12}
n_l \sum_{j=1}^{\infty} \frac{\eta_{lj}^2}{j^2} \le n_l m_l^{-2} \sum_{j=m_l}^{\infty} \eta_{lj}^2 \asymp n_l^{1-2r}m_l^{-2} \asymp C_l^{\frac{2r-1}{2r}}=   o(1) .
\end{equation}
By Theorem \ref{tcm},  (\ref{omu12}) implies  inconsistency of   sequence of alternatives $\tilde f_l$. \end{proof}

\begin{proof}[ Proof of Theorem \ref{tom5}] By Lemma \ref{lc1}, it suffices to prove that, for any $\varepsilon$, there is $n_0(\varepsilon)$ such that, for $n > n_0(\varepsilon)$, the following inequality holds
\begin{equation}\label{pl7}
\begin{split}&
|\mathbf{P}( T^2(b(F_n(t)+F_{1n}(t) - F_0(t)) +  \sqrt{n}(F_n(t) + F_{1n}(t) - 2 F_0(t))) > x_\alpha)\\& - \mathbf{P}( T^2(b(F_{n}(t)) +  \sqrt{n}(F_{n}(t) - F_0(t))) > x_\alpha)| < \varepsilon.
\end{split}
\end{equation}
Since $T$ is a norm, by Lemma \ref{lc4},  proof of (\ref{pl7}) is reduced to  proof that, for any $\delta_1 > 0$, there hold
\begin{equation}\label{pl9}
\mathbf{P}(|T(b(F_n(t)+F_{1n}(t) - F_0(t))) - T(b(F_{n}(t)))| > \delta_1) = o(1),
\end{equation}
and there is $\delta_n \to 0$ as  $n \to \infty$ such that there holds
\begin{equation}\label{pl11}
n^{1/2} |T(F_n(t)+F_{1n}(t) -2 F_0(t)) - T(F_{n}(t)- F_0(t))|  <\delta_n.
\end{equation}
Note that
\begin{equation}\label{pl13}
\begin{split}&
|T(b(F_n(t)+F_{1n}(t) - F_0(t))) - T(b(F_{n}(t)))|\\& \le T(b(F_n(t))+F_{1n}(t) - F_0(t)) - b(F_{n}(t)))
\end{split}
\end{equation}
and
\begin{equation}\label{pl14}
 |T(F_n(t)+F_{1n}(t) -2 F_0(t)) - T(F_{n}(t)- F_0(t))| \le T (F_{1n}(t)- F_0(t)).
\end{equation}
By Lemma \ref{lc2}, we have
\begin{equation}\label{pl15}
\mathbf{E}\, T^2(b(F_n(t)+F_{1n}(t) - F_0(t)) - b(F_{n}(t))) \le T^{1/4}(F_{1n} - F_0) =o(1).
\end{equation}
By (\ref{pl13}) and (\ref{pl15}), we get (\ref{pl9}).

Since sequence $f_{1n}$ is inconsistent, we have
\begin{equation}\label{pl16}
 nT^2(F_{1n}(t)- F_0(t)) = o(1)
 \end{equation}
as $n \to \infty$. By (\ref{pl14}) and (\ref{pl16}),  we get (\ref{pl11}). \end{proof}

\begin{proof}[ Proof of Theorem \ref{tom6}] We can write $f_n = f_{1n} + f_{2n}$ where $f_{1n} = \sum_{j < c k_n} \theta_{nj}  \, \phi_j$ and  $f_{2n} = \sum_{j  \ge c k_n} \theta_{nj}  \, \phi_j$. Denote $F_{1n}$ and $F_{2n}$ c.d.f.'s having densities $1 + f_{1n}$ and $1 + f_{2n}$ respectively. Then, using the inequality
\begin{equation}\label{pll7}
|T(F_n - F_0) - T(F_{1n} - F_0)| < T(F_n - F_{1n})
\end{equation}
and the same estimates as in proof of Theorem \ref{tom5} we get Theorem \ref{tom6}. We omit detailed reasoning. \end{proof}

Theorem \ref{tcm}, G1  and  B reduce  proof of Theorem \ref{tom4}  to the analysis of sums $\sum_{ck_n < j < Ck_n} \theta_{nj}^2$ with $C > c$. Such an analysis  has been provided in details in subsection \ref{subsec9.4} with another parameters $r$ and $s$. We omit  proof of Theorem \ref{tom4}.


\begin{thebibliography}{99}

\bibitem{an} Anderson, T. (1955)  The integral of  a symmetric unimodal function. { \it Proc.Amer.Math.Soc.} {\bf 6}(1) 170-176.

\bibitem{au}  Autin, F., Clausel,M., Jean-Marc Freyermuth, J. and  Marteau  C. (2018). Maxiset point of view for signal detection in inverse problems. arxiv 1803.05875.


\bibitem{chib} Chibisov, D.M. (1965)  An investigation of the asymptotic power of tests of fit. {\it Theor.Prob. Appl.} {\bf 10} 421 -437.

\bibitem{co} Cohen, A., DeVore, R., Kerkyacharian, G. and Picard, D. (2001). Maximal spaces with given rate
of convergence for thresholding algorithms, {\it Appl. Comput. Harmon. Anal.} {\bf 11} 167 – 191

\bibitem{dal}  Comminges, L. and  Dalalyan, A.S. (2013). Minimax testing of a composite null hypothesis defined via a quadratic functional in the model of regression. {\it Electron.J.Stat.} {\bf 7} 146-190.

 \bibitem{en} Engl, H., Hanke, M. and Neubauer, A. (1996). Regularization of Inverse Problems. Kluwer Academic
Publishers.

\bibitem{er90} Ermakov, M.S. (1990) Minimax detection of a signal in a Gaussian white noise.
{\it Theory Probab. Appl.}, {\bf 35} 667-679.

 \bibitem{er97}   Ermakov, M.S. (1997). Asymptotic minimaxity of chi-squared tests. {\it Theory Probab. Appl.} {\bf 42} 589–-610

  \bibitem{er03}   Ermakov, M.S. (2003). On asymptotic minimaxity of kernel-based tests. {\it  ESAIM Probab. Stat.} {\bf 7} 279–-312

\bibitem{er04} Ermakov, M.S.  (2006). Minimax detection of a signal in the  heteroscedastic Gaussian
white noise.
 {\it J. Math. Sci. (NY)},
{\bf 137}  4516-4524.

\bibitem{er15} Ermakov, M.S. (2017). On consistent hypothesis testing.  {\it J. Math. Sci. (NY)},
{\bf 225}  751-769.

 \bibitem{er18} Ermakov, M.S. (2018). On asymptotically minimax nonparametric detection of signal in Gaussian white noise. {\it Zapiski Nauchnih Seminarov POMI RAS.} {\bf 474} 124-138 (in Russian), arxiv.org 1705.07408.

    \bibitem{ih} Ibragimov,I.A. and Khasminskii, R.Z. (1977). On the estimation of infinitely dimensional parameter in Gaussian white noise. {\it Dokl.AN USSR} {\bf 236} 1053-1055.

\bibitem{ing87} Ingster, Yu.I. (1987). On comparison of the minimax properties of Kolmogorov, $\omega^2$ and $\chi^2$-tests. {\it Theory. Probab. Appl.} {\bf 32} 346-350.

\bibitem{ing02} Ingster,Yu.I.  and Suslina,I.A. (2002). {\it Nonparametric Goodness-of-fit Testing under Gaussian Models.}
Lecture Notes in Statistics {\bf 169} Springer: N.Y.

\bibitem{ing12}   Ingster,Yu. I.,  Sapatinas, T. and Suslina, I. A. (2012) {\it Minimax signal detection in ill-posed inverse problems}. --- {\it Ann. Statist.}, {\bf 40}  1524 – 1549.

\bibitem{jo}   Johnstone, I. M. (2015). Gaussian estimation. Sequence and wavelet models. {\it Book Draft} http://statweb.stanford.edu/~imj/

    \bibitem{ker93} Kerkyacharian, G. and Picard, D. (1993). Density estimation by kernel and wavelets methods: optimality of Besov spaces.
{\it Statist. Probab. Lett.} {\bf 18} 327 - 336.

\bibitem{ker02} Kerkyacharian, G. and Picard, D. (2002). Minimax or maxisets? {\it  Bernoulli} {\bf 8}, 219- 253.

\bibitem{la} Laurent, B., Loubes, J. M., and Marteau, C. (2011).
Testing inverse problems: a
direct or an indirect problem?
{\it J. Statist. Plann. Inference}
{\bf 141} 1849-–1861.

\bibitem{les} Le Cam, L. and Schwartz, L. (1960). A necessary and sufficient conditions for the existence of consistent estimates. {\it Ann.Math.Statist.} {\bf 31}  140-150.

    \bibitem{le73} Le Cam, L. (1973). Convergence of estimates under dimensionality restrictions. {\it Ann.Statist.} {\bf 1} 38-53.

        \bibitem{le} Lehmann, E.L. and Romano, J.P. (2005). {\it Testing Statistical Hypothesis}. Springer Verlag, NY.

\bibitem{lep}    Lepski, O.V. and Tsybakov, A.B.(2000). Asymptotically exact nonparametric hypothesis
testing in sup-norm and at a fixed point. {\it Probab. Theory Related
Fields}, {\bf 117}:1, 17–48.

\bibitem{rio}  Rivoirard, V. (2004). Maxisets for linear procedures.  {\it Statist. Probab. Lett.} {\bf 67}  267-275

\bibitem{wel} Shorack, G.R. and Wellner, J.A. (1986) Empirical Processes with Application to Statistics. J.Wiley Sons  NY

 \bibitem{sch} Schwartz, L. (1965). On Bayes procedures. {\it Z.Wahrsch.Verw. Gebiete} {\bf 4} 10-26.

\bibitem{ts} Tsybakov, A. (2009). {\it Introduction to Nonparametric Estimation.  }
Berlin: Springer.

\bibitem{ul} Ulyanov, P. L. (1964). {\it On  Haar series}. Mathematical Sbornik. {\bf 63(105)}:2 356-391. In Russian.
\end{thebibliography}
\end{document}